\documentclass[12pt]{article}
\usepackage{amsfonts}
\usepackage{amssymb}
\usepackage{mathrsfs}
\usepackage{srcltx}
\usepackage{amsmath}
\usepackage{amsthm}
\usepackage{amstext}
\usepackage{amsopn}
\usepackage{graphicx}
\usepackage{bm}
\newtheorem{theorem}{Theorem}[section]
\newtheorem{lemma}[theorem]{Lemma}
\newtheorem{proposition}[theorem]{Proposition}
\newtheorem{corollary}[theorem]{Corollary}

\theoremstyle{definition}

\theoremstyle{remark}
\newtheorem{remark}[theorem]{Remark}

\numberwithin{equation}{section}

\newcommand{\ba}{\begin{array}}
\newcommand{\ea}{\end{array}}
\newcommand{\f}{\frac}

\newcommand{\Om}{\Omega}

\newcommand{\la}{\lambda}

\newcommand{\ds}{\displaystyle}

\begin{document}
\date{}
\title{ \bf\large{Asymptotic profiles of the steady states for an SIS epidemic patch model with asymmetric connectivity matrix}\thanks{S. Chen is supported by National Natural Science Foundation of China (No 11771109) and a grant from China Scholarship Council,  J. Shi is supported by US-NSF grant DMS-1715651 and DMS-1853598, and Z. Shuai is supported by US-NSF grant DMS-1716445.}}
\author{Shanshan Chen\textsuperscript{1}\footnote{Email: chenss@hit.edu.cn},\ \ Junping Shi\textsuperscript{2}\footnote{Corresponding Author, Email: jxshix@wm.edu},\ \ Zhisheng Shuai\textsuperscript{3}\footnote{Email: shuai@ucf.edu}, \\ Yixiang Wu \textsuperscript{4}\footnote{Email: yixiang.wu@mtsu.edu} \\
{\small \textsuperscript{1} Department of Mathematics, Harbin Institute of Technology,\hfill{\ }}\\
\ \ {\small Weihai, Shandong, 264209, P.R.China.\hfill{\ }}\\
{\small \textsuperscript{2} Department of Mathematics, William \& Mary,\hfill{\ }}\\
\ \ {\small Williamsburg, Virginia, 23187-8795, USA.\hfill {\ }} \\
{\small \textsuperscript{3} Department of Mathematics, University of Central Florida,\hfill{\ }}\\
\ \ {\small Orlando, Florida, 32816, USA.\hfill {\ }} \\
{\small \textsuperscript{4} Department of Mathematics, Middle Tennessee State University,\hfill{\ }}\\
\ \ {\small Murfreesboro, Tennessee, 37132, USA.\hfill {\ }}
}
\maketitle

\begin{abstract}
The dynamics of an SIS epidemic patch model with asymmetric connectivity matrix is analyzed. It is shown that the basic reproduction number $R_0$ is strictly decreasing with respect to the dispersal rate of the infected individuals, and  the model has a unique endemic equilibrium if $R_0>1$. The asymptotic profiles of the endemic equilibrium for small dispersal rates are characterized. In particular, it is shown that the endemic equilibrium converges to a limiting disease-free equilibrium as the dispersal rate of susceptible individuals tends to zero, and
the limiting disease-free equilibrium has a positive number of susceptible individuals on each low-risk patch. Moreover a sufficient and necessary condition is found to  guarantee that the limiting disease-free equilibrium has no positive number of susceptible individuals on each high-risk patch. Our results extend earlier results for symmetric connectivity matrix, and we also partially solve an open problem by Allen et al. (SIAM J. Appl. Math., 67: 1283-1309, 2007).\\
{\bf Keywords}: SIS epidemic patch model,  asymmetric connectivity matrix, asymptotic profile\\
{\bf MSC 2010}: 92D30, 37N25, 92D40.
\end{abstract}

\section{Introduction}

Various mathematical models have been proposed to describe and simulate the transmissions of infectious diseases, and the predictions provided by those models may help to prevent and control the outbreak of the diseases \cite{1991book,2008book,Diekmann2000}. The spreading of the infectious diseases in populations depends on the  spatial structure of the environment and the dispersal pattern of the populations. The impact of the  spatial heterogeneity of the environment and the dispersal rate of the populations   on the transmission of the diseases can be modeled in  discrete-space settings by ordinary differential equation patch models \cite{Allenpatch,Arino2003,lloyd1996spatial,WangZhaoP1} or  in  continuous-space settings by reaction-diffusion equation models \cite{Allen,Fitz2008,WangZhao}.

In a discrete-space setting,
Allen \emph{et al}. \cite{Allenpatch} proposed the following SIS (susceptible-infected-susceptible) epidemic
patch model:
\begin{equation}\label{patc1}
\begin{cases}
\ds\f{d \overline S_j}{dt}=d_S\sum_{k\in\Omega} (\overline L_{jk}\overline S_k-\overline L_{kj}\overline S_j)-\ds\f{\beta_j\overline S_j\overline I_j}{\overline S_j+\overline I_j}+\gamma_j \overline I_j, &j\in\Omega,\\
\ds\f{d \overline I_j}{dt}=d_I\sum_{k\in\Omega}(\overline  L_{jk}\overline I_k-\overline L_{kj}\overline I_j)+\ds\f{\beta_j\overline S_j\overline I_j}{\overline  S_j+\overline I_j}-\gamma_j \overline I_j,&j\in\Omega,
\end{cases}
\end{equation}
where $\Omega=\{1,2,\dots,n\}$ with $n\ge2$. Here $\overline S_j(t)$ and $\overline I_j(t)$ denote the number of the susceptible and infected individuals in patch $j$ at time $t$, respectively;  $\beta_j$ denotes the rate of disease transmission and $\gamma_j$ represents the rate of disease recovery in patch $j$; $d_S,d_I$ are the dispersal rates of the susceptible and infected populations, respectively; and $\overline L_{jk}\ge0 $ describes the degree of the movement of the individuals from patch $k$ to patch $j$ for $j,k\in\Omega$. A major assumption in \cite{Allenpatch} is that the  matrix $(\overline L_{jk})$ is symmetric.  In \cite{Allenpatch}, the authors defined the basic reproduction number $R_0$ of the model \eqref{patc1}; they showed  that if $R_0<1$ the disease-free equilibrium is globally asymptotically stable, and if $R_0>1$ the model has a unique positive endemic equilibrium. Moreover, the asymptotic profile of the endemic equilibrium as $d_S\rightarrow 0$ is characterized in \cite{Allenpatch}, and the case $d_I\rightarrow 0$ is studied in \cite{LiPeng} recently. We remark that there are extensive studies on patch epidemic models, see \cite{AM2019,Eisenberg2013,GaoRuan,gao2019habitat,JinWang,LiShuai,LiShuai2,Driess,Tien2015,WangZhaoP1,WangZhaoP2} and the references therein. The corresponding reaction-diffusion model of \eqref{patc1} was studied in \cite{Allen} where the dispersal of the population is modeled by diffusion. A similar model with diffusive and advective movement of the population is studied in \cite{CuiLamLou,CuiLou}, and more studies on diffusive SIS models can be found in \cite{DengWu,Jiang,KuoPeng,LiPeng,Li2018,WebWu2018,Peng2009,PengLiu,Peng2013,Tuncer2012,WuZou,Wu2017} and the references therein.

The assumption that the  matrix $(\overline L_{jk})$ is symmetric in \cite{Allenpatch,LiPeng} is similar to the assumption of diffusive dispersal in reaction-diffusion models. However, asymmetric (e.g. advective) movements of the populations in space are common, and so in this paper we consider \eqref{patc1} with $(\overline L_{jk})$ being asymmetric and establish the corresponding results in \cite{Allenpatch,LiPeng}. Moreover, we will provide solutions to some of the open problems in \cite{Allenpatch} without assuming  $(\overline L_{jk})$ is symmetric: (1) we prove that the basic reproduction number $R_0$ is strictly decreasing in $d_I$; (2) we partially characterize the asymptotic profile of the $S$-component of the endemic equilibrium as $d_S\rightarrow 0$. The monotonicity of $R_0$ has also been proven recently in \cite{Chen2019,Gao2019, GaoDong} with $\beta_i, \gamma_i>0$ for all $i\in \Omega$, while this assumption will be dropped in our result.
We also establish  the asymptotic profile of  the endemic equilibrium as $d_I\rightarrow 0$ when $L$ is asymmetric, which extends the results of \cite{LiPeng} in which $L$ is assumed to be symmetric.

Denote
\begin{equation*}
L_{jk}=\begin{cases}
\overline L_{jk},&j\ne k,\\
-\ds\sum_{k\in\Omega,k\ne j} \overline L_{kj},&j=k,
\end{cases}
\end{equation*}
where $L_{jj}$ is the total degree of movement out from patch $j\in\Omega$. We rewrite \eqref{patc1} as:
\begin{equation}\label{patc}
\begin{cases}
\ds\f{d \overline S_j}{dt}=d_S\sum_{k\in\Omega} L_{jk}\overline S_k-\ds\f{\beta_j\overline S_j\overline I_j}{\overline S_j+\overline I_j}+\gamma_j \overline I_j, &j\in\Omega,\\
\ds\f{d \overline I_j}{dt}=d_I\sum_{k\in\Omega}L_{jk}\overline I_k+\ds\f{\beta_j\overline S_j\overline I_j}{\overline  S_j+\overline I_j}-\gamma_j \overline I_j,&j\in\Omega.
\end{cases}
\end{equation}
Let
\begin{equation*}
H^{-}=\{j\in\Omega: \beta_j<\gamma_j\} \;\;\text{and}\;\; H^{+}=\{j\in\Omega: \beta_j>\gamma_j\},
\end{equation*}
and $H^-$ and $H^+$ are referred as the sets of patches of low-risk and high-risk, respectively.
We impose the following four assumptions:
\begin{enumerate}
\item [($A_0$)] $\beta_j\ge 0$ and $\gamma_j\ge 0$ for all $j\in\Omega$; $d_S, d_I>0$;

\item [($A_1$)] The connectivity matrix $L:=(L_{jk})$ is irreducible and quasi-positive;
\item [($A_2$)] $\overline S_j(0)\ge0$, $\overline I_j(0)\ge0$, and \begin{equation}\label{N}
N:=\sum_{j\in\Omega} [\overline S_j(0)+\overline I_j(0)]>0;
\end{equation}
\item [($A_3$)] $H^{-}$ and $H^{+}$ are nonempty, and $\Omega=H^{-}\cup H^{+}$.
\end{enumerate}
By adding the $2n$ equations in \eqref{patc}, we see that the total population
is conserved in the sense that
\begin{equation}
N=\sum_{j\in\Omega}\left[ \overline S_j(t)+\overline I_j(t) \right]\;\;\text{for any}\;\;t\ge0.
\end{equation}
We remark that $(A_0)$-$(A_3)$ are assumed in \cite{Allenpatch} with  $L$ being symmetric in addition.

Throughout the paper, we use the following notations. For $n\ge2$,
\begin{equation}
\begin{split}
&\mathbb R^n=\{\bm u=(u_1,\dots, u_n)^T:u_i\in\mathbb R \;\;\text{for any} \;\;i=1,\dots,n\},\\
&\mathbb R_+^n=\{\bm u=(u_1,\dots, u_n)^T:u_i\ge0 \;\;\text{for any} \;\;i=1,\dots,n\}.
\end{split}
\end{equation}
For an $n\times n$ real-valued matrix $A$, we denote the spectral bound of $A$ by
$$s(A):=\max\{\mathcal Re (\la):\la \;\;\text{is an eigenvalue of}\;\;A\}, $$
and the spectral radius of $A$ by
$$\rho(A):=\max\{|\la|:\la \;\;\text{is an eigenvalue of}\;\;A\}. $$
The matrix $A$ is called nonnegative if all the entries of $A$ are nonnegative.
The matrix $A$ is called positive if $A$ is nonnegative and nontrivial.
The matrix $A$ is called zero if all the entries of $A$ are zero.
The matrix $A$ is called quasi-positive (or cooperative) if all off-diagonal entries of $A$ are nonnegative.

Let $\bm u=(u_1,\dots, u_n)^T$ and $\bm v=(v_1,\dots, v_n)^T$ be two vectors.
We write $\bm u\ge \bm v$ if $u_i\ge v_i$ for any $i=1,\dots,n$.
We write $\bm u>\bm v$ if $u_i\ge v_i$ for any $i=1,\dots,n$, and there exists
$i_0$ such that $u_{i_0}>v_{i_0}$.
We write $\bm u\gg \bm v$ if $u_i>v_i$ for any $i=1,\dots,n$.
We say $\bm u$ is strongly positive if $\bm u\gg \bm 0$.

The rest of the paper is organized as follows. In Section 2, we prove that model \eqref{patc}-\eqref{N} admits a unique endemic equilibrium if $R_0>1$ and prove that $R_0$ is strictly decreasing in $d_I$.
In Section 3, we study the asymptotic profile of the endemic equilibrium as $d_S\rightarrow 0$ and $d_I\rightarrow 0$, and we partially solve an open problem
in \cite{Allenpatch}. In Section 4, we consider an example where the patches form a star graph.

\section{The basic reproduction number}

In this section, we study the properties of the basic reproduction number $R_0$ of model \eqref{patc}.
The following result on the spectral bound of the connectivity matrix $L$  follows directly from the Perron-Frobenius theorem.
\begin{lemma}\label{simp}
 Suppose that ($A_1$) holds. Then $s(L)=0$ is a simple
eigenvalue of $L$ with a strongly positive eigenvector $\bm \alpha$, where \begin{equation}\label{alp}
\bm {\alpha}=(\alpha_1,\dots,\alpha_n)^T, \;\;\alpha_j>0 \;\;\text{for any}\;\; j\in\Omega,\;\text{and}\;\; \sum_{i=1}^n\alpha_i=1.
\end{equation} Moreover, there exists no other eigenvalue of $L$ corresponding with a nonnegative eigenvector.
\end{lemma}

In the rest of the paper, we denote $\bm{\alpha}$ the positive eigenvector of $L$ as specified in Lemma \ref{simp}.


Then we observe that model \eqref{patc}-\eqref{N} admits
a unique disease-free equilibrium.
\begin{lemma}\label{disf}
 Suppose that $(A_0)$-$(A_2)$ hold. Then model \eqref{patc}-\eqref{N} has a unique disease-free equilibrium $(\hat S_1,\dots,\hat S_n,0,\dots,0)^T$ with $\hat S_j=\alpha_jN$.
\end{lemma}
\begin{proof}
If $(\hat S_1,\dots,\hat S_n,0,\dots,0)^T$ is a disease-free equilibrium, then $$L(\hat S_1,\dots,\hat S_n)^T=0.$$ It follows from Lemma \ref{simp} that
there exists $\hat k\in \mathbb R$ such that $\hat S_j=\alpha_j \hat k$ for any $j\in\Omega$. Noticing that
\begin{equation*}
\sum_{j\in\Omega} S_j=\hat k \sum_{j\in\Omega}\alpha_j=N,
\end{equation*}
 we have $\hat k=N$. This completes the proof.
\end{proof}

We adopt the standard processes in \cite{Driessche} to compute the new infection and transition matrices:
\begin{equation}\label{FV}
F=diag(\beta_j),\;\;V=diag(\gamma_j)-d_I L,
\end{equation}
and the basic reproduction number $R_0$ is defined as
$$
R_0=\rho(FV^{-1}).
$$

We recall the following well-known result (see, e.g., \cite[ Corollary 2.1.5]{berman1994nonnegative}):
\begin{lemma}\label{lam}
Suppose that $P$ and $Q$ are $n\times n$ real-valued matrices, $P$ is quasi-positive, $Q$ is nonnegative and nonzero, and
$P+Q$ is irreducible. Then, $s(P+aQ)$ is strictly increasing for $a\in(0,\infty)$.
\end{lemma}

By Lemma \ref{lam}, if $\gamma_j$ ($j\in \Omega$) are not all zero, then $V$ is invertible and therefore an $M$-matrix. The following result follows from \rm \cite{Driessche}.
\begin{proposition}
 Suppose that $(A_0)$-$(A_1)$ hold and $\gamma_j$ ($j\in \Omega$) are not all zero. Then the following statements hold:
\begin{enumerate}
\item [(i)] $R_0-1$ has the same sign as $s(F-V)=s\left(d_I L+diag(\beta_j-\gamma_j)\right)$.
\item [(ii)] If $R_0<1$, the disease-free equilibrium $(\hat S_1,\dots,\hat S_n,0,\dots,0)^T$ of \eqref{patc}-\eqref{N} is locally asymptotically stable; if $R_0>1$, the disease-free equilibrium is unstable.
\end{enumerate}
\end{proposition}

The following result on the monotonicity of the spectral bound was proved in \cite[Theorem 3.3 and 4.4]{Chen2019}, which is related Karlin's theorem on the reduction principle in evolution biology \cite{altenberg2012resolvent, Karlin}.
\begin{lemma}\label{monos}
Suppose that $(A_1)$ holds. Let $f_j\in\mathbb{R}$ for $j\in\Omega$. Then the following two statements hold:
\begin{enumerate}
\item [$(i)$]
If $(f_1,\dots,f_n)$ is a multiple of $(1,\dots,1)$, then $s\left(d_I L+diag(f_j)\right)\equiv f_1$.
\item [$(ii)$] If $(f_1,\dots,f_n)$ is not a multiple of $(1,\dots,1)$, then
$s\left(d_I L+diag(f_j)\right)$ is strictly decreasing for $d_I\in(0,\infty)$.
\end{enumerate}
Moreover,
$$
\lim_{d_I\rightarrow 0} s\left(d_I L+diag(f_j)\right)= \max_{j\in \Omega} f_j,
$$
and
$$
\lim_{d_I\rightarrow \infty} s\left(d_I L+diag(f_j)\right)=  \sum_{j\in \Omega} f_j\alpha_j.
$$
\end{lemma}

Now we prove the monotonicity of the basic reproduction number $R_0$ with respect to $d_I$. We note that this result was also proved in \cite{Gao2019, GaoDong} with an additional assumption $\beta_j, \gamma_j>0$ for all $j\in\Omega$. If  $\gamma_j=0$, we set $\beta_j/\gamma_j=\infty$ when $\beta_j>0$ and $\beta_j/\gamma_j=0$ when $\beta_j=0$.
\begin{theorem}\label{samemono}
 Suppose that $(A_0)$-$(A_1)$ hold and $\gamma_j$ ($j\in \Omega$) are not all zero.
Then $R_0$ is strictly decreasing for $d_I\in(0,\infty)$ if $(\beta_1, \beta_2, ..., \beta_n)$ is not a multiple of $(\gamma_1, \gamma_2, ..., \gamma_n)$.
\end{theorem}
\begin{proof}
Clearly, $R_0=R_0(d_I)>0$ for $d_I \in(0,\infty)$. We claim that
\begin{equation}\label{R0bound}
    \ds\min_{j\in\Omega}\ds\f{\beta_j}{\gamma_j}\le R_0\le \ds\max_{j\in\Omega}\ds\f{\beta_j}{\gamma_j}.
\end{equation}
To see this, we first assume $\gamma_j>0$ for all $j\in\Omega$. Then, we have
$F_1\le F\le F_2$,
where
$$F_1=\left(\min_{j\in\Omega}\ds\f{\beta_j}{\gamma_j}\right)diag(r_j),\;\;F_2=\left(\max_{j\in\Omega}\ds\f{\beta_j}{\gamma_j}\right)diag(r_j).$$
Therefore,
\begin{equation}\label{lowupp}
\rho(F_1V^{-1})\le \rho(FV^{-1})\le \rho(F_2V^{-1}),
\end{equation}
 where $F$ and $V$ are defined by \eqref{FV}.
Since
\begin{equation}
\begin{split}
&(1,\dots,1)V=(\gamma_1,\dots,\gamma_n),\;\;(1,\dots,1)F_1=\left(\min_{j\in\Omega}\ds\f{\beta_j}{\gamma_j}\right)(\gamma_1,\dots,\gamma_n),\\
&(1,\dots,1)F_2=\left(\max_{j\in\Omega}\ds\f{\beta_j}{\gamma_j}\right)(\gamma_1,\dots,\gamma_n),\\
\end{split}
\end{equation}
we have
\begin{equation*}
\rho(F_1V^{-1})=\min_{j\in\Omega}\ds\f{\beta_j}{\gamma_j},\;\;\rho(F_2V^{-1})=\max_{j\in\Omega}\ds\f{\beta_j}{\gamma_j}.
\end{equation*}
This, together with \eqref{lowupp}, implies \eqref{R0bound}. It is not hard to check that \eqref{R0bound} still holds when $\gamma_j\ge 0$. Indeed, if $\gamma_{j_0}=\beta_{j_0}=0$ for some ${j_0}\in\Omega$, the arguments above still hold as $\beta_{j_0}/\gamma_{j_0}=0$. If $\gamma_{j_0}=0$ and $\beta_{j_0}>0$ for some ${j_0}\in\Omega$, then  $\beta_{j_0}/\gamma_{j_0}=\infty$. We can replace the $j_0$-th entry of $F_1$ by $0$ to obtain the first inequality of \eqref{R0bound}, and the second inequality of \eqref{R0bound} is trivial.

Let \begin{equation}\label{mu0}
\mu_0(d_I)=\frac{1}{R_0(d_I)},
\end{equation}
and  $$\la_1(d_I,a):=s(-V+aF)=s\left(d_IL+a F-diag(\gamma_j)\right).$$
The following discussion is divided into two cases.\\
{\bf Case 1.}
For any $a\in(0,\infty)$, $(a\beta_1-\gamma_1,\dots,a\beta_n-\gamma_n)$ is not a multiple of $(1,\dots,1)$. Then we see from Lemma \ref{monos} that
for any fixed $a>0$, $\la_1(d_I, a)$ is strictly decreasing for $d_I\in(0,\infty)$.
 Let $\phi>0$ be the corresponding eigenvector of $V^{-1}F$ with respect to $\rho(V^{-1}F)$. Then
\begin{equation*}
d_I L\phi-diag(\gamma_j)\phi+\mu_0(d_I)F\phi=0.
\end{equation*}
Since $L$ is irreducible, it follows that
$\phi\gg 0$ and $\la_1(d_I,\mu_0(d_I))=0$ for any $d_I>0$.
Let $d_I^{1}>d_I^{2}$.
Then
\begin{equation}\label{mumon}
\begin{split}
&\la_1\left(d_I^{2},\mu_0\left(d_I^{1}\right)\right)-\la_1\left(d_I^{2},\mu_0\left(d_I^{2}\right)\right)\\
=&\la_1\left(d_I^{2},\mu_0\left(d_I^{1}\right)\right)-\la_1\left(d_I^{1},\mu_0\left(d_I^{1}\right)\right)>0,
\end{split}
\end{equation}
which implies that
$$\mu_0\left(d_I^{1}\right)>\mu_0\left(d_I^{2}\right),$$
and consequently, $\mu_0(d_I)$ is strictly increasing for $d_I\in(0,\infty)$.\\
{\bf Case 2.} There exists $\tilde a>0$ such that $(\tilde a\beta_1-\gamma_1,\dots,\tilde a\beta_n-\gamma_n)$ is a multiple of $(1,\dots,1)$.
That is, there exists $k\in\mathbb {R}$ such that
$$
(\tilde a\beta_1-\gamma_1,\dots,\tilde a\beta_n-\gamma_n)=k(1,\dots,1).
$$
Clearly, $\tilde a$ is unique and $k\ne 0$ if $(\beta_1, \beta_2, ..., \beta_n)$ is not a multiple of $(\gamma_1, \gamma_2, ..., \gamma_n)$.

If $k>0$, then $\beta_j>0$ for all $j\in\Omega$ and
$$
R_0\ge \ds\min_{j\in\Omega}\ds\f{\beta_j}{\gamma_j}>\ds\f{1}{\tilde a},
$$
which implies that $\mu_0(d_I)<\tilde a$ for any $d_I>0$.
Similar to Case 1, it follows from Lemma \ref{monos} that $\la_1(d_I, a)$ is strictly decreasing for $d_I\in(0,\infty)$ for any fixed $a<\tilde a$.
Therefore, \eqref{mumon} holds, and
$\mu_0(d_I)$ is strictly increasing for $d_I\in(0,\infty)$.

If $k<0$, then $\gamma_j>0$ for all $j\in\Omega$ and
$$
R_0\le \ds\max_{j\in\Omega}\ds\f{\beta_j}{\gamma_j}<\ds\f{1}{\tilde a},
$$
which implies that $\mu_0(d_I)>\tilde a$ for any $d_I>0$. The
rest of the proof is similar to the case of $k>0$.
\end{proof}

 Then we compute the limits of $R_0$ as $d\to0$ or $d\to\infty$.
\begin{theorem}\label{ass}
Suppose that $(A_0)$-$(A_1)$ hold and $\gamma_j$ ($j\in \Omega$) are not all zero. Then the basic reproduction number $R_0=R(d_I)$ satisfies the following:
\begin{equation*}
\lim_{d_I\to0} R_0(d_I)=\ds\max_{j\in\Omega}\ds\f{\beta_j}{\gamma_j}\;\;\text{and}\;\;\lim_{d_I\to\infty} R_0(d_I)=\f{\sum_{j\in\Omega}\alpha_j\beta_j}{\sum_{j\in\Omega} \alpha_j\gamma_j}.
\end{equation*}
\end{theorem}
\begin{remark}
 If $L$ is symmetric, then
 $$
 \lim_{d_I\to\infty} R_0(d_I)=\f{\sum_{j\in\Om}\beta_j}{\sum_{j\in\Om}\gamma_j}.
 $$
 as $\alpha_j\equiv 1/n$.
\end{remark}
\begin{proof}
Let $\mu_0(d_I)$ and $\lambda_1(d_I, a)$ be defined as in the proof of Theorem \ref{samemono}. Noticing that $\mu_0(d_I)$ is increasing in $d_I$, let
$$
\mu_1=\lim_{d_I\to0} \mu_0(d_I) \ \ \text{and }\ \ \mu_2=\lim_{d_I\to\infty} \mu_0(d_I),
$$
where $\mu_1\in [0, \infty)$ and  $\mu_2\in (0, \infty]$.
By Lemma \ref{monos}, for any $a>0$,
\begin{equation}\label{limit11}
\lim_{d_I\to0} \lambda_1(d_I, a)=\max_{j\in\Omega}\{a\beta_j-\gamma_j \} \ \ \text{and }\ \ \lim_{d_I\to\infty} \lambda_1(d_I, a)=\sum_{j\in\Omega}(a\beta_j-\gamma_j)\alpha_j.
\end{equation}
Since  $\la_1\left(d_I,\mu_0\left(d_I\right)\right)=0$, we have
\begin{equation}\label{limit22}
\max_{j\in\Omega}\{\mu_1\beta_j-\gamma_j \}=0 \ \ \text{and }\ \ \sum_{j\in\Omega}(\mu_2\beta_j-\gamma_j)\alpha_j=0.
\end{equation}
Indeed, to see the first equality, for given $\epsilon>0$ there exists $\hat d_I>0$ such that $\mu_1-\epsilon<\mu_0(d_I)<\mu_1+\epsilon$ for all $d_I<\hat d_I$. By Lemma \ref{lam}, we have
$$
\la_1\left(d_I,\mu_1-\epsilon\right)< \la_1\left(d_I,\mu_0\left(d_I\right)\right)=0 <\la_1\left(d_I,\mu_1+\epsilon\right) \ \ \text{for all } d_I<\hat d_I.
$$
By \eqref{limit11}, we have
\begin{equation*}
\max_{j\in\Omega}\{(\mu_1-\epsilon)\beta_j-\gamma_j \} \le 0 \le \max_{j\in\Omega}\{(\mu_1+\epsilon)\beta_j-\gamma_j \}.
\end{equation*}
Since $\epsilon>0$ is arbitrary, we obtain the first equality. The other equality in \eqref{limit22} can be proved similarly.

It follows from \eqref{limit22} that
\begin{equation*}
\lim_{d_I\to0} R_0(d_I)\ge \ds\max_{j\in\Omega}\ds\f{\beta_j}{\gamma_j} \;\;\text{and}\;\;\lim_{d_I\to\infty} R_0(d_I)=\f{\sum_{j\in\Omega}\alpha_j\beta_j}{\sum_{j\in\Omega} \alpha_j\gamma_j},
\end{equation*}
where the equality holds for $d_I\to 0$ if there exists no $j\in\Omega$ such that $\beta_j=\gamma_j=0$.
Noticing \eqref{R0bound}, the proof is complete.
\end{proof}

\section{The endemic equilibrium}
In this section, we consider the endemic equilibrium of model \eqref{patc}-\eqref{N}. The equilibria of \eqref{patc}-\eqref{N} satisfy
\begin{equation}\label{patcs}
\begin{cases}
d_S\ds\sum_{k\in\Omega} L_{jk}   S_k-\ds\f{\beta_j   S_j   I_j}{   S_j+   I_j}+\gamma_j    S_j=0, &j\in\Omega,\\
d_I\ds\sum_{k\in\Omega} L_{jk}   I_k+\ds\f{\beta_j   S_j   I_j}{    S_j+   I_j}-\gamma_j    I_j=0,&j\in\Omega,\\
\ds\sum_{j\in\Omega} (   S_j+   I_j)=N.
\end{cases}
\end{equation}
Firstly, we study the existence and uniqueness of the endemic equilibrium. Then, we investigate the asymptotic profile of the endemic equilibrium as $d_S\rightarrow 0$ and/or $d_I\rightarrow0$.

\subsection{The existence and uniqueness}
In this section, we show that \eqref{patc}-\eqref{N} has a unique endemic equilibrium if $R_0>1$. Motivated by \cite{Allenpatch}, we first introduce an equivalent problem of \eqref{patcs}.
\begin{lemma}\label{equiv1}
Suppose that $(A_0)$-$(A_3)$ hold. Then $(S_1,\dots, S_n, I_1,\dots, I_n)^T$ is a non-negative solution of \eqref{patcs} if and only if
\begin{equation*}
( S_1,\dots, S_n, I_1,\dots, I_n)= \left(\kappa \check S_1,\dots,\kappa \check S_n, \ds\f{\kappa}{d_I} \check I_1,\dots,\ds\f{\kappa}{d_I} \check I_n\right),
 \end{equation*}
where $( \check S_1,\dots, \check S_n, \check I_1,\dots, \check I_n)$ satisfies
\begin{equation}\label{equiveq}
\begin{cases}
d_S \check S_j+ \check I_j=\alpha_j, &j\in\Omega,\\
d_I\ds\sum_{k\in\Omega}L_{jk}\check I_k+\check I_j\left(\beta_j-\gamma_j-\ds\f{d_S\beta_j \check I_j}{d_I(\alpha_j-\check I_j)+d_S\check I_j}\right)=0,&j\in\Omega,\\
\end{cases}
\end{equation}
and
\begin{equation}\label{kap}
\kappa=\ds\f{d_IN}{\ds\sum_{j\in\Omega}(d_I\check S_j+\check I_j)}.
\end{equation}
\end{lemma}
\begin{proof}
Clearly, from \eqref{patcs}, we have
\begin{equation*}
\sum_{k\in\Omega}L_{jk}\left(d_S     S_k+d_I    I_k\right)=0\;\;\text{for any}\;\;j\in\Omega.
\end{equation*}
Then it follows from Lemma \ref{simp} that there exists $\kappa>0$ such that
\begin{equation}\label{kape}
d_S    S_j+d_I    I_j=\kappa \alpha_j\;\;\text{for any}\;\; j\in\Omega.
\end{equation}
Let
\begin{equation}\label{subs}
\check S_j=\ds\f{    S_j}{\kappa},\;\; \check I_j=\ds\f{d_I    I_j}{\kappa}.
\end{equation}
Then $$d_S \check S_j+\check  I_j=\alpha_j\;\; \text{for any}\;\; j\in\Omega.$$
Plugging \eqref{kape}-\eqref{subs} into the second equation of \eqref{patcs}, we see that $\check I_j$ satisfies the second equation of
\eqref{equiveq}. Since
\begin{equation*}
N=\sum_{j\in\Omega}(    S_j+    I_j)=\kappa\sum_{j\in\Omega}\left(\check S_j+\ds\f{\check I_j}{d_I}\right),
\end{equation*}
\eqref{kap} holds.
This completes the proof.
\end{proof}

From Lemma \ref{equiv1}, to analyze the solutions of \eqref{equiveq}, we only need to consider the equations of $\check I_j$ in \eqref{equiveq}.
We consider an auxiliary problem of \eqref{equiveq}.

\begin{lemma}\label{exis}
Suppose that $(A_0)$-$(A_3)$ hold and $R_0>1$. Then, for any $d>0$, the following equation
\begin{equation}\label{au}
\begin{cases}
d_I\ds\sum_{k\in\Omega} L_{jk} \check I_k+\check I_j\left(\beta_j-\gamma_j-\ds\f{\beta_j\check I_j}{d(\alpha_j-\check I_j)+\check I_j}\right)=0,&j\in\Omega,\\
0\le \check I_j\le \alpha_j, &j\in\Omega,
\end{cases}
\end{equation}
admits exactly one non-trivial solution $\bm {\check I}=(\check I_1,\dots, \check  I_n)^T$, where $0< \check I_j<\alpha_j$ for any $j\in\Omega$. Moreover, $\check I_j$ is  monotone increasing in $d\in (0, \infty)$ for any $j\in\Omega$.
\end{lemma}
\begin{proof}
Since $R_0>1$, $s\left(d_I L+diag(\beta_j-\gamma_j)\right)>0$.
Let
\begin{equation}\label{fj}f_j(\check I_j)=\check I_j\left(\beta_j-\gamma_j-\ds\f{\beta_j\check I_j}{d(\alpha_j-\check I_j)+\check I_j}\right),
\end{equation}
and consider the following problem
\begin{equation}\label{au2}
\frac{d \bar I_j}{dt}=d_I\sum_{k\in\Omega} L_{jk} \bar I_k+f_j(\bar I_j),\;\;j\in\Omega,\;\;t>0.
\end{equation}
Let $\bm g(\bm {\check I})=\left(g_1(\bm {\check I}),\dots,g_n(\bm {\check I})\right)^T$ be the vector field corresponding to the right hand side of \eqref{au2}, and let
\begin{equation*}
U=\{\bm {\check I}=\left(\check I_1,\dots,\check I_n\right)^T\in\mathbb R^n:0\le \check I_j\le \alpha_j,\;j\in\Omega\}.
\end{equation*}
Then $U$ is positive invariant with respect to \eqref{au2}, and for any $\bm {\check I}\in U$,
\begin{equation*}
D_{\bm {\check I}} \bm g(\bm {\check I})=d_IL+diag(f'_j(\check I_j)),
\end{equation*}
which is irreducible and quasi-positive.
Let $\Psi_t$ be the semiflow induced by \eqref{au2}. By \cite[Theorem B.3]{SmithW}, $\Psi_t$ is strongly positive and monotone.

For any $\la\in(0,1)$ and $\check I_j\in (0,\alpha_j]$, we have
\begin{equation}
\begin{split}
f_j(\la \check I_j)-\la f(\check I_j)=&-\ds\f{\la^2\beta_j \check I_j^2}{d(\alpha_j-\la \check I_j)+\la \check I_j}+\ds\f{\la\beta_j \check I_j^2}{d(\alpha_j- \check I_j)+ \check I_j}\\
=&\ds\f{d\la\alpha_j\beta_j\check I_j^2(1-\la)}{[d(\alpha_j-\la \check I_j)+\la \check I_j][d(\alpha_j- \check I_j)+ \check I_j]}\ge0,
\end{split}
\end{equation}
and the strict inequality holds for at least one $j$.
This implies that $\bm g(\bm {\check I})$ is strictly sublinear on $U$ (see \cite{ZhaoJing} for the definition of strictly sublinear functions).
Noticing $s\left(d_I L+diag(\beta_j-\gamma_j)\right)>0$, it follows from \cite[Theorem 2.3.4]{ZhaoBook} (or \cite[Corollary 3.2]{ZhaoJing}) that there exists a unique  $\bm {\check I}\gg\bm 0$ in $U$ such
that every solution in $U\setminus \{\bm 0\}$ converges to $\bm {\check I}$.
Moreover, if $\check I_j=\alpha_j$ for some $j\in\Omega$, then
$\check I'_j\le -\gamma_j<0$, which implies that
$\check I_j\in (0,\alpha_j)$ for any $j\in\Omega$.

Suppose $d_1>d_2$. Let $\bm {\check I}^{(i)}=(\check I_1^{(i)},\dots,\check  I_n^{(i)})^T$ be the unique strongly positive solution of \eqref{au} with $d=d_i$ for $i=1,2$,
 and let $ \bm {\bar I}^{(i)}(t)=(\bar I_1^{(i)}(t),\dots, \bar I_n^{(i)}(t))^T$  be the solution of \eqref{au2} with $d=d_i$ for $i=1,2$, and $\bm {\bar I}^{(1)}(0)=\bm {\bar I}^{(2)}(0)\in U\setminus \{\bm 0\}$. Then for any $j\in\Omega$,
\begin{equation}
\begin{split}
\ds\f{d \bar I_j^{(1)}}{dt}=&d_I\sum_{k\in\Omega} L_{jk}  \bar I_k^{(1)}+ \bar I^{(1)}_j\left(\beta_j-\gamma_j-\ds\f{\beta_j \bar I^{(1)}_j}{d_1(\alpha_j- \bar I^{(1)}_j)+ \bar I^{(1)}_j}\right)\\
\ge&d_I \sum_{k\in\Omega} L_{jk}  \bar I_k^{(1)} + \bar I^{(1)}_j\left(\beta_j-\gamma_j-\ds\f{\beta_j \bar I^{(1)}_j}{d_2(\alpha_j- \bar I^{(1)}_j)+ \bar I^{(1)}_j}\right).
\end{split}
\end{equation}
It follows from the comparison principle that
$\bar I_j^{(1)}(t)\ge  \bar I_j^{(2)}(t)$ for any $t\ge0$ and $j\in\Omega$. Therefore,
$\check I_j^{(1)}=\ds\lim_{t\to\infty} \bar I_j^{(1)}(t)\ge   \check I_j^{(2)}=\ds\lim_{t\to\infty} \bar I_j^{(2)}(t)$ for any $j\in\Omega$.
\end{proof}

Lemma \ref{exis} was proved in \cite{Allenpatch} when $L$ is symmetric by virtue of the upper and lower solution method. Here we prove it without assuming the symmetry of $L$ by the monotone dynamical system method.

By Lemmas \ref{equiv1}-\ref{exis},  we can show that model \eqref{patc}-\eqref{N} has a unique endemic equilibrium if $R_0>1$.
\begin{theorem}\label{mainexis}
Suppose that $(A_0)$-$(A_3)$ hold and $R_0>1$.  Then  \eqref{patc}-\eqref{N} has exactly two non-negative equilibria: the disease-free equilibrium and the endemic equilibrium
\begin{equation}\label{ende}
(  S_1,\dots,  S_n,  I_1,\dots,  I_n)=\left(\kappa  \check S_1,\dots, \kappa \check S_n, \ds\f{\kappa  \check I_1}{d_I},\dots,\ds\f{\kappa  \check I_n}{d_I}\right),
\end{equation}
where
\begin{equation}\label{skap}
\check  S_j=\ds\f{\alpha_j- \check I_j}{d_S},\;\; \kappa=\ds\f{d_IN}{\ds\sum_{j\in\Omega}(d_I \check S_j+ \check I_j)},
\end{equation}
and $(\check I_1,\dots, \check I_n)^T$ is the unique strongly positive solution of \eqref{au} with $d:={d_I}/{d_S}$.
\end{theorem}
\begin{proof}
This result follows from Lemmas \ref{equiv1}--\ref{exis}.
\end{proof}

\subsection{Asymptotic profile with respect to $d_S$}
In this subsection, we study the asymptotic profile of the endemic equilibrium of \eqref{patc}-\eqref{N} as $d_S\to 0$.
We suppose that $(A_0)$-$(A_3)$ hold throughout this subsection.
Moreover, we observe that $R_0$ is independent of $d_S$. Therefore, we assume  $R_0>1$ throughout this subsection so that the endemic equilibrium exists for all $d_S>0$.

We first study the asymptotic profile of $\kappa$ and $I_j$, where $\kappa$ and $I_j$ are defined in Theorem \ref{mainexis}.
\begin{lemma}\label{II}
As $d_S\to0$, $\kappa\to 0$ and $I_j\to 0$ for any $j\in\Omega$.
\end{lemma}
\begin{proof}
For any sequence $\{d_S^{(m)}\}_{m=1}^\infty$ such that $\ds\lim_{m\to \infty} d_S^{(m)}=0$, we denote the corresponding endemic equilibrium by $(   S^{(m)}_1,\dots,   S^{(m)}_n,   I^{(m)}_1,\dots,   I^{(m)}_n)$.
Since $   I^{(m)}_j\in(0,N]$, there exists a subsequence
$\{d_S^{(m_l)}\}_{l=1}^\infty$ such that
$\ds\lim_{l\to\infty}   I_j^{(m_l)} =   I^*_j$ for some $   I^*_j\in [0, N]$. For $j\in H^{-}$,
$$d_S^{(m_l)}\sum_{k\in\Omega}L_{jk}   S^{(m_l)}_k\le    I_j^{(m_l)}(\beta_j-\gamma_j)\le0.$$
Since $   S^{(m_l)}_k\in(0,N]$ for any $l\ge1$ and $k\in\Omega$, we have
$$\lim_{l\to\infty}d_S^{(m_l)}\sum_{k\in\Omega}L_{jk}   S^{(m_l)}_k=0,$$
which implies $   I_j^*=0$. Therefore $   I_j\to 0$ as $d_S\to0$ for $j\in H^{-}$.

Since
 $$d_S   S_j+d_I I_j=\kappa \alpha_j\;\;\text{for any}\;\; j\in\Omega,$$
 and $H^{-}\ne\emptyset$ by $(A_3)$,
 we have $\kappa\to 0$ as $d_S\to0$. This in turn implies that for $j\in H^+$,
$$
 I_j=\ds\f{\kappa\alpha_j-d_S S_j}{d_I}\to 0\;\;\text{as}\;\;d_S\to0.
$$
\end{proof}

\begin{lemma}\label{mono}
 For each $j\in\Omega$, $\check I_j$ is monotone decreasing in $d_S\in (0, \infty)$ and $\ds\lim_{d_S\to 0}
\check I_j=\check I_j^*\in(0,\alpha_j].$
\end{lemma}
\begin{proof}
We notice that  $(\check I_j)$ is the positive solution of \eqref{au} with $d=d_I/d_S$. By Lemma \ref{exis},
$\check I_j$ is monotone increasing  in $d$, which implies that
$\check I_j$ is monotone decreasing in $d_S$ for each $j\in\Omega$. Since $\check I_j\in(0,\alpha_j)$ from Lemma \ref{equiv1}, we have   $\ds\lim_{d_S\to 0}\check I_j=\check I_j^*\in(0,\alpha_j].$
\end{proof}

From Lemma \ref{mono},  we denote
\begin{equation}\label{jpjn}
J^-=\{j\in\Omega: 0<\check I_j^*<\alpha_j\},\;\;\text{and}\;\;J^+=\{j\in\Omega:\check I^*_j=\alpha_j\}.
\end{equation}
Clearly $\Omega=J^-\cup  J^+$.  We show that $J^-$ is nonempty.
\begin{lemma}
$J^-$ is nonempty, and $H^{-}\subset J^{-}$.
\end{lemma}
\begin{proof}
Suppose that there exists $j\in\Omega$ such that $\beta_j-\gamma_j<0$ and
$\check I_j^*=\alpha_j$. By \eqref{au}, we have
$$
d_I\sum_{k\in\Omega} L_{jk} \check I_k+\check I_j(\beta_j-\gamma_j)\ge0.
$$
Taking $d_S\to0$ on both sides, we have
\begin{equation}\label{ineq}
d_I\sum_{k\ne j,k\in\Omega}L_{jk} \check I_k^*+d_IL_{jj}\alpha_j\ge \alpha_j(\gamma_j-\beta_j)>0.
\end{equation}
Since $$d_I\sum_{k\ne j,k\in\Omega}L_{jk} \alpha_k+d_IL_{jj}\alpha_j=0,$$ and
$\check I_j^*\in(0,\alpha_j]$ for any $j\in\Omega$, we have
$$d_I\sum_{k\ne j,k\in\Omega}L_{jk} \check I_k^*+d_IL_{jj}\alpha_j\le0,$$
which contradicts with \eqref{ineq}. Therefore, $H^{-}\subset J^{-}$.
\end{proof}

By virtue of the above lemma, we can prove the following result about the asymptotic profile of $S_j$. The proof is similar to \cite[Lemma 4.4]{Allenpatch}, and we omit it here.
\begin{lemma}\label{estimi}
Let $J^{-}$ be defined as above. Then
\begin{enumerate}
\item [$(i)$]$\ds\lim_{d_S\to 0}\f{\kappa}{d_S}=\ds\f{N}{\ds\sum_{k\in J^{-}}(\alpha_k-\check I_k^*)}$;
\item [$(ii)$]For any  $j\in\Omega$, $\ds\lim_{d_S\to 0}S_j= \ds\f{N}{\ds\sum_{k\in J^-}(\alpha_k-\check I_k^*)}(\alpha_j-\check I_j^*)$.
\end{enumerate}
\end{lemma}

Similar to \cite[Lemma 4.5]{Allenpatch}, we can prove that $J^+$ is nonempty.
\begin{lemma}
$J^+$ is nonempty.
\end{lemma}

For some further analysis of $J^+$ with respect to $d_I$, we define
 \begin{equation}\label{M}
M=\left(M_{jk}\right)_{j,k\in H^-},\;\; \text{where}\;\; M_{jk}=\begin{cases}
-d_IL_{jk},&j,k\in H^{-},\;j\ne k,\\
-d_IL_{jj}-(\beta_j-\gamma_j), &j,k\in H^{-},\;j= k,
\end{cases}
\end{equation}
Then $M$ is an $M$-matrix, and $M^{-1}$ is positive. Therefore,
the following system
\begin{equation}\label{sovH}
-d_I\sum_{k\in H^{-}}L_{jk}I_k-(\beta_j-\gamma_j)I_j=d_I\sum_{k\in H^+}L_{jk} \alpha_k,\;\;j\in H^{-},
\end{equation}
has a unique solution $(I_j)_{j\in H^-}=\left(\alpha_{j}^*\right)_{j\in H^{-}}$.

Define \begin{equation}\label{alpj}
\check I_j^{(0)}=\begin{cases}
\alpha_j^*,&j \in H^{-},\\
\alpha_j, &j\in H^{+},
\end{cases}
\end{equation}
and denote
\begin{equation}\label{hj}
h_j(d_I)=d_I\ds\sum_{k\in\Omega} L_{jk} \check I_k^{(0)}+(\beta_j-\gamma_j)\alpha_j, \ \ j \in H^+.
\end{equation}

We have the following result on the asymptotic profile of the endemic equilibrium as $d_S\to 0$.
\begin{theorem}\label{dsmain}
Suppose that $(A_0)$-$(A_3)$ hold and $R_0>1$. Let $(S_1,\dots, S_n,  I_1,\dots, I_n)$ be the unique endemic equilibrium of \eqref{patc}-\eqref{N} and
$\bm {\check I}=(\check I_1,\dots,\check I_n)^T$ be the unique strongly positive solution of \eqref{au} with $d={d_I}/{d_S}$.
Then the following statements hold:
\begin{enumerate}
\item [$(i)$] 	$\ds\lim_{d_S\to0} (S_1,\dots, S_n,  I_1,\dots, I_n)=(S_1^*,\dots,S_n^*,0,\dots,0)$.
\item [$(ii)$] If $h_j(d_I)>0$ for all $j\in H^+$,
then $J^+=H^+$ and $J^-=H^-$. Moreover,
\begin{equation}\label{S*}
 S_j^*=
\left\{
\begin{array}{lll}
\ds\f{\alpha_j-\alpha_j^*}{\ds\sum_{k\in H^-}(\alpha_k-\alpha_k^*)} N,\ \ & \text{for}\;\; j\in H^{-}, \\
0,  &\text{for}\;\; j\in H^{+}.
\end{array}
\right.
\end{equation}
\item [$(iii)$] If $h_{j_0}(d_I)<0$ for some $j_0\in H^+$ and $h_j(d_I)\ne0$ for any $j\in H^+$,
then $H^{-}\subsetneqq J^{-}$ and $J^{+}\subsetneqq H^+$.
Moreover, there exists $j_1\in H^+$ such that $\ds S_{j_1}^*>0$ and
$\ds S_j^*>0$ for any $j\in H^{-}$.
\end{enumerate}
\end{theorem}

\begin{proof}
(i) follows from Lemma \ref{II}. Without loss of generality, we assume $H^{-}=\{1,2,\dots,p\}$ and $H^{+}=\{p+1,\dots,n\}$ for some $p>0$. Then
\begin{equation*}
 \check I_j^{(0)}=\begin{cases}
\alpha^*_j,&1\leq j\leq p,\\
\alpha_j,&p+1\leq j \leq n,
\end{cases}
\end{equation*}
and $M=\left(M_{jk}\right)_{1\le j,k\le p}$  is defined as in \eqref{M}.
Since
\begin{equation}\label{suppli}
\begin{split}
&-\left[d_I\sum_{k=1}^p L_{jk}\alpha_k+(\beta_j-\gamma_j)\alpha_j\right]>d_I\sum_{k=p+1}^n L_{jk}\alpha_k\;\;\text{for}\;\;1\leq j\leq p,\\
&-\left[d_I\sum_{k=1}^p L_{jk}\alpha^*_k+(\beta_j-\gamma_j)\alpha^*_j\right]=d_I\sum_{k=p+1}^n L_{jk}\alpha_k\;\;\text{for}\;\;1\leq j\leq p,
\end{split}
\end{equation}
and $M^{-1}$ is positive, we have $\alpha^*_j\in[0,\alpha_j)$ for any $1\leq j\leq p$. Since $L$ is irreducible, it is not hard to show that $\alpha^*_j>0$ for any $1\leq j\leq p$.


Define
\begin{equation*}
\bm G(d_S,\bm {\tilde I})=\left(\begin{array}{c}
\left[d_I\sum_{k\in \Omega} L_{1k}\tilde I_k+(\beta_{1}-\gamma_1)\tilde I_1\right]\left[d_S\tilde I_1+d_I(\alpha_1-\tilde  I_1)\right]-d_S\beta_1 \tilde I_1^2 \\
\left[d_I\sum_{k\in \Omega} L_{2k}\tilde I_k+(\beta_{2}-\gamma_2)\tilde  I_2\right]\left[d_S \tilde I_2+d_I(\alpha_2-\tilde  I_2)\right]-d_S\beta_2 \tilde  I_2^2 \\
\vdots\\
\left[d_I\sum_{k\in \Omega} L_{nk}\tilde I_k+(\beta_{n}-\gamma_n)\tilde I_n\right]\left[d_S\tilde I_n+d_I(\alpha_n-\tilde I_n)\right]-d_S\beta_n \tilde I_n^2 \\
\end{array}\right),
\end{equation*}
where $\bm {\tilde I}=(\tilde I_1,\dots, \tilde I_n)^T$.
Let $\bm {\check I}^{(0)}=(\check I_1^{(0)},\dots,\check I_n^{(0)})$. Then $\bm G(0, \bm {\check I}^{(0)})=\bm 0$. Moreover, if \eqref{au} has a solution $\bm {\check I}$ with $d={d_I}/{d_S}$, then
$\bm G(d_S,\bm  {\check I})=\bm 0$; if $\bm G(d_S,\bm  {\check I})=\bm 0$ with $\bm {\check I}=({\check I}_1,\dots, \check I_n)^T$ satisfying $0<{\check I}_j< \alpha_j$, then $\bm {\check I}$ is a nontrivial solution of \eqref{au} with $d={d_I}/{d_S}$.

A direct computation shows that
$$
D_{\bm {\tilde I}}G\left(0,\bm {\check I}^{(0)}\right)=(V_{jk})_{j, k\in\Omega},
$$ where
\begin{equation*}
V_{jk}=\begin{cases}
d_I^2(\alpha_j-\alpha_j^*)L_{jk},&1\leq j\leq p, \; k\ne j,\\
d_I(\alpha_j-\alpha_j^*)\left(d_IL_{jj}+(\beta_j-\gamma_j)\right),&1\leq j\leq p,\; k=j,\\
0,&p+1\leq j\leq n,\; k \ne j,\\
-d_I\left[d_I\ds\sum_{k\in\Omega} L_{jk}{\check I}_k^{(0)}+(\beta_j-\gamma_j)\alpha_j\right],&p+1\leq j\leq n,\;k=j.
\end{cases}
\end{equation*}
Therefore, we have
$$
D_{\bm {\tilde I}}G\left(0,\bm {\check I}^{(0)}\right)=
d_I
\begin{pmatrix}
V_1 & * \\
\bm 0 & V_2
\end{pmatrix}
$$
where $V_1$ is a $p\times p$ matrix
$$
V_1=
\begin{pmatrix}
(\alpha_1-\alpha_1^*)(d_IL_{11}+\beta_1-\gamma_1) & d_I(\alpha_1-\alpha_1^*)L_{12}& \cdots &  d_I(\alpha_1-\alpha_1^*)L_{1p} \\
\cdots & \cdots & \cdots & \cdots\\
d_I(\alpha_p-\alpha_p^*)L_{p1} & d_I(\alpha_p-\alpha_p^*)L_{p2}& \cdots &  (\alpha_p-\alpha_p^*)(d_IL_{1p}+\beta_p-\gamma_p)
\end{pmatrix}
$$
and $V_2=\text{diag}(-h_j(d_I))$ is a diagonal matrix. It is not hard to check that $V_1$ is non-singular. Indeed, $V_1$ has negative diagonal entries and nonnegative off-diagonal entries. Moreover, the sum of the $j$-th row of $V_1$ is
\begin{equation*}
\begin{split}
&d_I\sum_{k=1}^pL_{jk}\alpha_j+(\beta_j-\gamma_j)\alpha_j-d_I\sum_{k=1}^pL_{jk}\alpha_j^*-(\beta_j-\gamma_j)\alpha_j^*\\
=&d_I\sum_{k=1}^n L_{jk}\alpha_j+(\beta_j-\gamma_j)\alpha_j
=(\beta_j-\gamma_j)\alpha_j<0,
\end{split}
\end{equation*}
where we used \eqref{sovH} and Lemma \ref{simp}. Therefore, $V_1$ is strictly diagonally dominant and invertible ($-V_1$ is an $M$-matrix).
Hence if $h_j(d_I)\ne 0$ for any $j\in H^+$,  $(V_{jk})$ is invertible. It follows from the implicit function theorem
that there exist a constant $\delta>0$, a neighborhood $N(\bm {\check I}^{(0)})$ of $\bm {\check I}^{(0)}$ and a continuously differentiable function
$$
\bm {\tilde I}\left(d_S\right)=(\tilde I_1\left(d_S\right),\dots,\tilde I_n\left(d_S\right))^T:[0,\delta]\to N(\bm {\check I}^{(0)})
$$
such that for any $d_S\in [0,\delta]$, the unique solution of $\bm G(d_S,\bm{\tilde I})=\bm 0$ in the neighborhood $N(\bm {\check I}^{(0)})$ is $\tilde {\bm I}\left(d_S\right)$ and $\tilde {\bm I}\left(0\right)=\bm {\check I}^{(0)}$.

Differentiating $\bm G(d_S,\bm {\tilde I}(d_S))=\bm 0$ with respect to $d_S$ at $d_S=0$, and using the definition of  ${\check I}_j^{(0)}$, we have
\begin{equation*}
\begin{cases}
d_I(\alpha_j-\alpha_j^*)\left[d_I\ds\sum_{k\in\Omega} L_{jk} \tilde I_k'(0)+(\beta_j-\gamma_j)\tilde I'_j(0)\right]-\beta_j(\alpha_j^*)^2=0,\;\;\;&1\leq j\leq  p,\\
-d_I\left[d_I\ds\sum_{k\in\Omega} L_{jk}\check I_k^{(0)}+(\beta_j-\gamma_j)\alpha_j\right]\tilde I'_j(0) &\\
=-d_I\alpha_j\ds\sum_{k\in\Omega} L_{jk}\check I_k^{(0)}+\gamma_j\alpha_j^2>0,\;\;\;&p+1\leq j\leq n.
\end{cases}
\end{equation*}
If $h_j(d_I)>0$ for all $j\in H^{+}$, then $\tilde I_j'(0)<0$ for every $j\in H^+$. This implies that $\tilde I_j(d_S)\approx \alpha_j+\tilde I_j'(0)d_S<\alpha_j$ for $j\in H^{+}$ if $d_S>0$ is sufficiently small. Moreover for $j\in H^{-}$, $\tilde I_j(d_S)\approx \alpha_j^*<\alpha_j$ for small $d_S>0$. Therefore,  $\bm{\tilde I}$ is a nontrivial solution of \eqref{au}, and $\bm{\tilde I}=\bm{\check I}$ by the uniqueness of the positive solution of \eqref{au}. Since $\ds\lim_{d_S\to0}\bm{\check I}=\bm {\check I}^{(0)}$,
we have $J^+=H^+$ and $J^-=H^-$. By Lemma \ref{estimi}, we have
 $$
 \ds S_j^*=\lim_{d_S\to0} S_j=\ds\f{\alpha_j-\alpha_j^*}{\ds\sum_{k\in H^-}(\alpha_k-\alpha_k^*)}N\;\;\text{for}\;\; j\in H^{-},
 $$ and $S_j^*=\ds\lim_{d_S\to0} S_j= 0$ for $j\in H^+$.

On the other hand, if there exists $j_0\in H^+$ such that $h_{j_0}(d_I)<0$, then $\tilde I_{j_0}'(0)>0$, which implies that  $\tilde I_{j_0}(d_S)\approx \alpha_{j_0}+\tilde I_{j_0}'(0)d_S>\alpha_{j_0}$, so
$\bm {\tilde I}$ is not a solution of \eqref{au} with $d={d_I}/{d_S}$. Therefore, $\ds\lim_{d_S\to0}\bm{\check I}\ne\bm {\check I}^{(0)}$, which yields
$H^{-}\subsetneqq J^{-}$ and $J^{+}\subsetneqq H^+$. Then there exists $j_1\in H^+$ such that $\ds S_{j_1}^*>0$. This completes the proof.
\end{proof}

The function $h_j(d_I)$ in Theorem \ref{dsmain} is critical in determining the asymptotic profile of the endemic equilibrium as $d_S\to 0$. The next result explores further properties of the function $h_j(d_I)$.

\begin{proposition}\label{prop3.10}
Suppose that $(A_0)$-$(A_3)$ hold, $R_0>1$,  and $H^{-}=\{1,2,\dots,p\}$ and $H^{+}=\{p+1,\dots,n\}$ for some $p>0$. Then for any $p+1\le j\le n$, $h_j(d_I)$ is either constant or strictly decreasing in $d_I$. Moreover,
\begin{equation*}
    \lim_{d_I\rightarrow\infty}
\begin{pmatrix}
h_{p+1}(d_I) \\
\vdots\\
h_n(d_I)
\end{pmatrix}
=
-\tilde N\tilde{M}^{-1}
\begin{pmatrix}
(\gamma_1-\beta_1)\alpha_1 \\
\vdots\\
(\gamma_p-\beta_p)\alpha_p
\end{pmatrix}
+
\begin{pmatrix}
(\beta_{p+1}-\gamma_{p+1})\alpha_{p+1} \\
\vdots\\
(\beta_n-\gamma_n)\alpha_n
\end{pmatrix},
\end{equation*}
and
\begin{equation*}
    \lim_{d_I\rightarrow 0}
\begin{pmatrix}
h_{p+1}(d_I) \\
\vdots\\
h_n(d_I)
\end{pmatrix}
=
\begin{pmatrix}
(\beta_{p+1}-\gamma_{p+1})\alpha_{p+1} \\
\vdots\\
(\beta_n-\gamma_n)\alpha_n
\end{pmatrix},
\end{equation*}
where $\tilde M=(\tilde m_{ij})$ is a $p\times p$ matrix with $\tilde m_{ij}=-L_{ij}$ for $1\le i, j\le n$ and $\tilde N=(\tilde n_{ij})$ is an $(n-p)\times p$ matrix with $\tilde n_{ij}=L_{(i+p)j}$ for $1\le i\le n-p$ and $1\le j\le p$, i.e.
$$
L=
\begin{pmatrix}
- \tilde M & *\\
\tilde N & *
\end{pmatrix}.
$$
\end{proposition}
\begin{proof}
First we claim that $\alpha_j^*$ is strictly increasing in $d_I$ for each $1\le j\le p$. To see this, we differentiate both sides of \eqref{sovH} with respect to $d_I$ to get
\begin{equation}\label{1111}
-d_I\sum_{k=1}^p L_{jk}(\alpha_k^*)'-(\beta_j-\gamma_j)(\alpha_j^*)'-\sum_{k=1}^p L_{jk}\alpha_k^*=\sum_{k=p+1}^n L_{jk}\alpha_k, \ \ \ 1\le j\le p.
\end{equation}
Combining \eqref{sovH} and \eqref{1111}, we have
$$
-d_I\sum_{k=1}^p L_{jk}(\alpha_k^*)'-(\beta_j-\gamma_j)(\alpha_j^*)'=d_I^{-1}(\gamma_j-\beta_j)\alpha_j^*>0, \ \ \ 1\le j\le p.
$$
Since $M$ is an $M$-matrix and $\beta_j<\gamma_j$ for $1\le j\le p$, $(\alpha_j^*)'$ is strictly positive. This proves the claim.

By the fact that $\alpha_j^*\in (0, \alpha_j)$  and the monotonicity of $\alpha_j^*$ for $d_I\in (0, \infty)$, the limits $\ds\lim_{d_I\rightarrow 0} \alpha_j^*$ and $\ds\lim_{d_I\rightarrow \infty} \alpha_j^*$ exist for $1\le j\le p$. It is not hard to see
$$
\lim_{d_I\rightarrow 0} \alpha_j^*=0.
$$
Dividing both sides of \eqref{sovH} by $d_I$ and taking $d_I\rightarrow\infty$, we have
$$
-\sum_{k=1}^p L_{jk}\lim_{d_I\rightarrow \infty} \alpha_k^*=\sum_{k=p+1}^n L_{jk}\alpha_k, \ \ \ 1\le j\le p.
$$
Therefore,
\begin{equation}\label{alphalimit}
\lim_{d_I\rightarrow \infty} \alpha_j^*=\alpha_j,  \ \ \ 1\le j\le p.
\end{equation}

Next  we claim that $\alpha_j^*+ d_I (\alpha_j^*)'<\alpha_j$ for all $1\le j\le p$ and $d_I>0$. To see this, by \eqref{1111}, we have
$$
-\sum_{k=1}^p L_{jk} (\alpha_k^*+d_I(\alpha_k^*)')>\sum_{k=p+1}^n L_{jk}\alpha_k, \ \ \ 1\le j\le p.
$$
By the definition of $\alpha_j$,
\begin{equation}\label{alp111}
-\sum_{k=1}^p L_{jk}\alpha_k=\sum_{k=p+1}^n L_{jk}\alpha_k, \ \ \ 1\le j\le p.
\end{equation}
Then the claim follows from the fact that $\tilde M$ is an $M$-matrix.

Differentiating $h_j(d_I)$ with respect to $d_I$, we find
$$
h'_j(d_I)=\sum_{k=1}^p L_{jk} (\alpha_k^*+d_I(\alpha_k^*)') + \sum_{k=p+1}^n L_{jk} \alpha_k, \ \ \ p+1\le j\le n.
$$
It follows from \eqref{alp111} that
$$
h'_j(d_I)=\sum_{k=1}^p L_{jk} (\alpha_k^*+d_I(\alpha_k^*)'-\alpha_k), \ \ \ p+1\le j\le n.
$$
Since $\alpha_j^*+ d_I (\alpha_j^*)'<\alpha_j$ for all $1\le j\le p$, either $h'_j(d_I)<0$ or $h'_j(d_I)=0$ for all $d_I>0$ and $p+1\le j\le n$. Therefore, $h_j(d_I)$ is either strictly decreasing or constant for all $d_I>0$ and $p+1\le j\le n$.

Finally, we compute the limit of $h_j(d_I)$. By \eqref{sovH} and $L\bm \alpha=0$, we have
$$
-d_I\sum_{k=1}^p L_{jk}(\alpha_k-\alpha_k^*)-(\beta_j-\gamma_j)(\alpha_j-\alpha_j^*)=-(\beta_j-\gamma_j)\alpha_j, \ \ \ 1\le j\le p.
$$
Let $u_j=d_I(\alpha_j-\alpha_j^*)$, $1\le j\le p$. Then,
$$
-\sum_{k=1}^p L_{jk}u_k-\frac{(\beta_j-\gamma_j)}{d_I}u_j=-(\beta_j-\gamma_j)\alpha_j, \ \ \ 1\le j\le p.
$$
Taking $d_I\rightarrow\infty$, we find
$$
-\sum_{k=1}^p L_{jk}\lim_{d_I\rightarrow\infty} u_k=-(\beta_j-\gamma_j)\alpha_j, \ \ \ 1\le j\le p.
$$
So, we have
$$
\lim_{d_I\rightarrow\infty}
\begin{pmatrix}
u_1 \\
u_2 \\
\vdots\\
u_p \\
\end{pmatrix}
=\tilde M^{-1}
\begin{pmatrix}
(\gamma_1-\beta_1)\alpha_1 \\
(\gamma_2-\beta_2)\alpha_2 \\
\vdots\\
(\gamma_p-\beta_p)\alpha_p \\
\end{pmatrix}.
$$

Since
\begin{eqnarray*}
h_{j}(d_I)&=& d_I\sum_{k=1}^n L_{jk}\alpha_k+ d_I\sum_{k=1}^p  L_{jk}(\alpha_k^*-\alpha_k)+(\beta_j-\gamma_j)\alpha_j \\
&=& -\sum_{k=1}^p  L_{jk}u_k+(\beta_j-\gamma_j)\alpha_j
,\ \ \ p+1\le j\le n,
\end{eqnarray*}
we have
\begin{eqnarray*}
\lim_{d_I\rightarrow\infty}
\begin{pmatrix}
h_{p+1}(d_I) \\
h_{p+2}(d_I) \\
\vdots\\
h_{n}(d_I) \\
\end{pmatrix}
&=&-\tilde N
\lim_{d_I\rightarrow\infty}
\begin{pmatrix}
u_1 \\
u_2 \\
\vdots\\
u_p \\
\end{pmatrix}
+
\begin{pmatrix}
(\beta_{p+1}-\gamma_{p+1})\alpha_{p+1} \\
(\beta_{p+2}-\gamma_{p+2})\alpha_{p+2} \\
\vdots\\
(\beta_{n}-\gamma_n)\alpha_n \\
\end{pmatrix} \\
&=&
-\tilde N\tilde M^{-1}
\begin{pmatrix}
(\gamma_1-\beta_1)\alpha_1 \\
(\gamma_2-\beta_2)\alpha_2 \\
\vdots\\
(\gamma_p-\beta_p)\alpha_p \\
\end{pmatrix}
+
\begin{pmatrix}
(\beta_{p+1}-\gamma_{p+1})\alpha_{p+1} \\
(\beta_{p+2}-\gamma_{p+2})\alpha_{p+2} \\
\vdots\\
(\beta_{n}-\gamma_n)\alpha_n \\
\end{pmatrix}.
\end{eqnarray*}
The limit of $h_j(d_I)$ as $d_I\rightarrow 0$ follows from \eqref{alphalimit} and the definition of $h_j(d_I)$.
\end{proof}
Now we have the following results summarizing the dynamics of \eqref{patc}-\eqref{N} when the diffusion rate of the infectious population $d_I$ varies and the diffusion rate of the susceptible population $d_S$ tends to $0$.

\begin{corollary}\label{cor:3.12}
Suppose that $(A_0)$-$(A_3)$ hold and $R_0>1$. Let $(S_1^*,\dots,S_n^*,0,\dots,0)$ be the limiting disease-free equilibrium as $d_S \to 0$ defined as in Theorem \ref{dsmain}.
Then there exists $d_I^*\in (0,\infty]$ and $d_I^{**}\in (0,d_I^*]$ such that
\begin{enumerate}
\item when $0<d_I<d_I^*$, $R_0(d_I)>1$ and there exists a unique endemic equilibrium $(S_1,\dots,S_n,I_1,\dots,I_n)$ of  \eqref{patc}-\eqref{N}; and when $d_I>d_I^*$, $R_0(d_I)<1$ and the disease-free equilibrium is globally asymptotically stable.
    \item When $0<d_I<d_I^{**}$, $H^+=J^+$ and $H^-=J^-$; and $S_j^*>0$ for $j\in H^-=J^-$, $S_j^*=0$ for $j\in H^+=J^+$ as defined in \eqref{S*}.
    \item When $d_I^{**}<d_I<d^{*}$ and except a finite number of $d_I$'s, $H^+=J^+\cup J^-_{1}$, $H^-=J^-_{2}$, where $J^-=J^-_1\cup J^-_2$ such that $J^-_{1}\ne \emptyset$; and $S_j^*>0$ for $j\in J^-$, $S_j^*=0$ for $j\in J^+$.
\end{enumerate}
\end{corollary}
\begin{proof}
From the condition ($A_3$) and Theorem \ref{ass}, $R_0>1$ for small $d_I>0$. From the monotonicity of $R_0$ shown in Theorem \ref{samemono}, either (i) there exists a unique $d_I^*>0$ such that $R_0(d_I)=1$ and when $R_0>1$ when $d_I>d_I^*$, or (ii) $R_0>1$ for all $d_I>0$. We denote $d_I^*=\infty$ in the  case (ii). The uniqueness of endemic equilibrium is shown in Theorem \ref{mainexis}, and the global stability of the disease-free equilibrium when $R_0<1$ has been shown in \cite{Allenpatch}.

For $0<d_I<d_I^*$, $h_j(d_I)>0$ for all $j\in H^+$ and small $d_I>0$ from Proposition \ref{prop3.10}. Then from part $(ii)$ of Theorem \ref{dsmain}, for $d_I>0$ small, $H^+=J^+$ and $H^-=J^-$; and $S_j^*>0$ for $j\in H^-=J^-$, $S_j^*=0$ for $j\in H^+=J^+$ as defined in \eqref{S*}. From the monotonicity of $h_j(d_I)$ shown in Proposition \ref{prop3.10}, either (i) there exists a unique $d_I^{**}\in(0,d_I^*)$ such that $h_j(d_I)>0$ for all $j\in H^+$ and $d\in (0,d_I^{**})$ and $h_{j_0}(d_I^{**})=0$ for some $j_0\in H^+$, or (ii) $h_j(d_I)>0$ for all $j\in H^+$ and $d\in (0,d_I^{*})$. We let $d_I^{**}=d_I^*$ in case (ii). In case (i), the monotonicity of $h_{j_0}(d_I)$ implies that $h_{j_0}(d_I)<0$ for all $d_I\in (d_I^{**},d_I^*)$, and except a finite number of $d_I$'s, $h_j(d_I)\ne 0$ for $d_I\in (d_I^{**},d_I^*)$. Thus results in part $(iii)$ of Theorem \ref{dsmain} hold for all $d_I\in (d_I^{**},d_I^*)$ except a finite number of $d_I$'s.
\end{proof}

We show that the condition on the function $h_j(d_I)$ is comparable to the conditions on $d_I$ given in \cite{Allenpatch}.

\begin{proposition}\label{3.12}
Suppose that $(A_0)$-$(A_3)$ hold and $L$ is symmetric. Define
\begin{equation}
    L_k^-=\sum_{j\in H^-, \; j\neq k} L_{kj}, \;\;\; L_k^+=\sum_{j\in H^+, \; j\neq k} L_{kj}.
\end{equation}
If
\begin{equation}\label{1.6}
\frac{1}{d_I}> \max_{k\in H^+} \frac{L_k^-}{\beta_k-\gamma_k}+\max_{k\in H^-} \frac{L_k^+}{\beta_k-\gamma_k},
\end{equation}
then $h_j(d_I)>0$ for all $j\in H^+$.
\end{proposition}
\begin{proof}
Assume on the contrary that  $h_j(d_I)\le 0$ for some $j\in H^+$.
Let $\alpha_m^*=\min\{\alpha_k^*: k\in H^-\}$. Since $L$ is symmetric,  $\alpha_j=1/n$ for all $j\in \Omega$. Then, we have
\begin{equation}\label{mid11}
h_j(d_I)=d_I\sum_{k\in H^-}L_{jk}\alpha_k^*+\frac{1}{n}\left[d_I\sum_{k\in H^+}L_{jk}+\beta_j-\gamma_j\right]\le 0.
\end{equation}
Since $j\in H^+$ and $L_{jj}=-L_j^+-L_j^-$, we have $\ds\sum_{k\in H^+}L_{jk}=-L_j^-$. Therefore, by \eqref{mid11} and the definition of $\alpha_m^*$, we have
$$
d_IL_j^-\alpha_m^*+\frac{1}{n}\left[-d_IL_j^-+\beta_j-\gamma_j\right]\le 0,
$$
which implies
\begin{equation}\label{mid22}
    n\alpha_m^*\le \frac{\gamma_j-\beta_j+d_IL_j^-}{d_IL_j^-}.
\end{equation}

By $m\in H^-$ and \eqref{sovH}, we have
$$
d_I\sum_{k\in H^-}L_{mk}\alpha_k^*+d_I\sum_{k\in H^+}L_{mk}\alpha_k+(\beta_m-\gamma_m)\alpha_m^*=0,
$$
which impiles
$$
d_I\sum_{k\in H^-, k\neq m}L_{mk}(\alpha_k^*-\alpha_m^*)-d_IL_m^+\alpha_m^*+d_I\frac{L_m^+}{n}+(\beta_m-\gamma_m)\alpha_m^*=0.
$$
By the definition of $\alpha_m^*$, we have
$$
-d_I L_m^+\alpha_m^*+d_I\frac{L_m^+}{n}+(\beta_m-\gamma_m)\alpha_m^*\le 0.
$$
Therefore,
$$
\frac{d_IL_m^+}{-\beta_m+\gamma_m+d_IL_m^+}\le n\alpha_m^*.
$$
It then follows from \eqref{mid22} that
$$
\frac{d_IL_m^+}{-\beta_m+\gamma_m+d_IL_m^+}\le \frac{\gamma_j-\beta_j+d_IL_j^-}{d_IL_j^-},
$$
which can be simplified as
$$
(\gamma_m-\beta_m)(\gamma_j-\beta_j)+(\gamma_j-\beta_j)d_IL_m^++(\gamma_m-\beta_m)d_IL_j^-\ge 0.
$$
Dividing both sides by $d_I(\gamma_m-\beta_m)(\gamma_j-\beta_j)$ (which is negative), we obtain
$$
\frac{1}{d_I}\le \frac{L_m^+}{\beta_m-\gamma_m}+\frac{L_j^-}{\beta_j-\gamma_j}\le \max_{j\in H^-}\frac{L_j^+}{\beta_j-\gamma_j}+\max_{j\in H^+}\frac{L_j^-}{\beta_j-\gamma_j},
$$
which is a contradiction. Therefore, $h_j(d_I)>0$ for all $j\in H^+$.
\end{proof}

\begin{remark}\begin{enumerate}
\item By Theorem \ref{dsmain}, the unique endemic equilibrium converges to a limiting disease-free equilibrium as $d_S\to0$. Moreover, the limiting
disease-free equilibrium has a positive number of susceptible individuals on each
low-risk patch. This is in agreement of the results in \cite{Allenpatch} which assumes $L$ is symmetric.
\item  In \cite{Allenpatch},
the distribution of susceptible individuals on high-risk patches  is left as an open problem. In Theorem \ref{dsmain}, we show that  the distribution of susceptible individuals on high-risk patches depends on the function $h_j(d_I)$:
$S_j^*=0$ on each high-risk patch if $h_j(d_I)>0$ for each high-risk patch $j$, and the monotonicity of $h_j(d_I)$ in $d_I$ shown in Proposition \ref{prop3.10} implies $S_j^*=0$ on each high-risk patch when $0<d_I<d_I^{**}$. This partially solves this open problem in \cite{Allenpatch}.
\item The sharp threshold diffusion rate $d_I^{**}$ is characterized by the smallest zero of function $h_j(d_I)$ on any high-risk patch $j$. When $L$ is symmetric, a lower bound of $d_I^{**}$ is shown in Proposition \ref{3.12} and also \cite[Theorem 2]{Allenpatch}:
\begin{equation}\label{1.6a}
d_I^{**}\ge  \left[\max_{k\in H^+} \frac{L_k^-}{\beta_k-\gamma_k}+\max_{k\in H^-} \frac{L_k^+}{\beta_k-\gamma_k}\right]^{-1}:=\widetilde{d_I^{**}}.
\end{equation}
It is an interesting question to have a more explicit expression or estimate of $d_I^{**}$ when $L$ is not symmetric.
\end{enumerate}
\end{remark}

\subsection{Asymptotic profile with respect to $d_I$ and $d_S$}

We suppose that $(A_0)$-$(A_3)$  hold throughout this subsection, and we consider the asymptotic profile of the  endemic equilibrium of \eqref{patc}-\eqref{N} as $d_I\to0$.
  The case that $L$ is symmetric was studied in \cite{LiPeng} recently, and we consider the  asymmetric case here. For simplicity, we assume $\gamma_j>0$ for any $j\in\Omega$.
 Since $\ds\lim_{d_I\to0} R_0=\max_{j\in\Omega}{\beta_j}/{\gamma_j}>1$,
we have $R_0>1$ ($s\left(d_I L+diag(\beta_j-\gamma_j)\right)>0$) and the existence and uniqueness of the endemic equilibrium for sufficiently small $d_I$.

Firstly, we consider the asymptotic profile of positive solution of \eqref{au} as $d_I\to 0$. We denote $(x)_+=0$ if $x\le 0$ and $(x)_+=x$ if $x>0$.
\begin{lemma}\label{asym}
Suppose that $(A_0)$-$(A_3)$ hold and $\gamma_j>0$ for all $j\in\Omega$.   Let $\bm {\check I}=(\check I_1,\dots,\check  I_n)^T$ be the unique strongly positive solution of \eqref{au}. Then the following two statements hold:
\begin{enumerate}
\item [$(i)$] For any $d>0$,
\begin{equation}\label{lims}
\lim_{d_I\to0} \check I_j=\ds\f{d\alpha_j\left(\beta_j-\gamma_j\right)_+}{d(\beta_j-\gamma_j)_++\gamma_j},\;\;j\in\Omega.
\end{equation}

\item [$(ii)$] As $(d_I,d)\to (0,\infty)$ (or equivalently, $(d_I,1/d)\to (0,0)$),
\begin{equation*}
\check I_j\to 0\;\;\text{for} \;\;j\in H^{-} \;\;\text{and} \;\;\check I_j\to \alpha_j\;\; \text{for} \;\;j\in H^{+}.
\end{equation*}
\end{enumerate}
\end{lemma}

\begin{proof}
$(i)$ Define
\begin{equation}
\bm F(d_I,\bm {\tilde I})=\left(\begin{array}{c}
d_I\sum_{k\in \Omega} L_{1k}{\tilde I}_k+{\tilde I}_1\left(\beta_1-\gamma_1-\ds\f{\beta_1{\tilde I}_1}{d(\alpha_1-{\tilde I}_1)+{\tilde I}_1}\right)\\
d_I\sum_{k\in \Omega} L_{2k}{\tilde I}_k+{\tilde I}_2\left(\beta_2-\gamma_2-\ds\f{\beta_2{\tilde I}_2}{d(\alpha_2-{\tilde I}_2)+{\tilde I}_2}\right)\\
\vdots\\
d_I\sum_{k\in \Omega} L_{nk}{\tilde I}_k+{\tilde I}_n\left(\beta_n-\gamma_n-\ds\f{\beta_n{\tilde I}_n}{d(\alpha_n-{\tilde I}_n)+{\tilde I}_n}\right)\\
\end{array}\right),
\end{equation}
and denote $\bm {\check I}^{(1)}=\left({\check I}^{(1)}_1,\dots,{\check I}_n^{(1)}\right)^T$, where
 $$
{\check I}_j^{(1)}=\ds\f{d\alpha_j\left(\beta_j-\gamma_j\right)_+}{d(\beta_j-\gamma_j)_++\gamma_j}\;\;\text{for}\;\;j\in\Omega.
$$
Clearly, $\bm F(0,\bm {\check I}^{(1)})=\mathbf 0$, and $D_{\bm {\tilde I}} \bm F(0,\bm {\check I}^{(1)})=diag(\delta^{(1)}_j)$, where
\begin{equation}
\delta^{(1)}_j=\begin{cases} \beta_j-\gamma_j<0,&j\in H^-,\\
-\ds\f{d\alpha_j\beta_j{\check I}_j^{(1)}}{\left[d\left(\alpha_j-{\check I}_j^{(1)}\right)+{\check I}_j^{(1)}\right]^2}<0,&j\in H^+.
\end{cases}
\end{equation}
Therefore, $D_{\bm {\tilde I}} \bm F(0,\bm {\check I}^{(1)})$ is invertible. It follows from the implicit function theorem that
there exist $ d_1>0$ and a continuously differentiable mapping
$$d_I\in[0, d_1]\mapsto \bm {{\tilde I}}(d_I)=({\tilde I}_1(d_I),\dots,{\tilde I}_n(d_I))^T\in\mathbb R^n$$ such that $\bm F(d_I,\bm {{\tilde I}}(d_I))=\mathbf 0$ and
$\bm {{\tilde I}}(0)=\bm {\check I}^{(1)}$.

Taking the derivative of $\bm F(d_I,\bm {{\tilde I}}(d_I))=\bm 0$ with respect to $d_I$ at $d_I=0$, we have
\begin{equation*}
-diag(\delta_j^{(1)})\left(\begin{array}{c} {\tilde I}'_1(0)\\
 {\tilde I}'_2(0)\\
 \vdots\\
  {\tilde I}'_n(0)\end{array}\right)=L\left(\begin{array}{c} {\tilde I}_1(0)\\
 {\tilde I}_2(0)\\
 \vdots\\
  {\tilde I}_n(0)\end{array}\right).
\end{equation*}
Then
\begin{equation*}
\left(\begin{array}{c} {\tilde I}'_1(0)\\
 {\tilde I}'_2(0)\\
 \vdots\\
  {\tilde I}'_n(0)\end{array}\right)=-diag(1/\delta_j^{(1)})L\left(\begin{array}{c} {\tilde I}_1(0)\\
 {\tilde I}_2(0)\\
 \vdots\\
  {\tilde I}_n(0)\end{array}\right).
\end{equation*}
Since  $\bm {\tilde I}(0)=\bm {\check I}^{(1)}>\bm 0$, we see that ${\tilde I}'_j(0)>0$ for $j\in H^-$, which implies that
$\bm {{\tilde I}}=\bm{\check I}$, and consequently,  \eqref{lims} holds.

$(ii)$ Let $\eta=1/d$. Define
\begin{equation*}
\bm H(d_I,\eta,\bm {\tilde I})=\left(\begin{array}{c}
\left[d_I\sum_{k\in \Omega} L_{1k}{\tilde I}_k+(\beta_{1}-\gamma_1){\tilde I}_1\right]\left[\alpha_1-{\tilde I}_1+\eta {\tilde I}_1\right]-\eta\beta_1 {\tilde I}_1^2 \\
\left[d_I\sum_{k\in \Omega} L_{2k}{\tilde I}_k+(\beta_{2}-\gamma_2){\tilde I}_2\right]\left[\alpha_2-{\tilde I}_2+\eta {\tilde I}_2\right]-\eta\beta_2 {\tilde I}_2^2 \\
\vdots\\
\left[d_I\sum_{k\in \Omega} L_{nk}{\tilde I}_k+(\beta_{n}-\gamma_n){\tilde I}_n\right]\left[\alpha_n-{\tilde I}_n+\eta {\tilde I}_n\right]-\eta \beta_n {\tilde I}_n^2 \\
\end{array}\right),
\end{equation*}
and denote $\bm {\check I}^{(2)}=({\check I}^{(2)}_1,\dots,{\check I}_n^{(2)})^T$, where
\begin{equation*}
{\check I}^{(2)}_j=\begin{cases}
0,& j\in H^{-},\\
\alpha_j,&j\in H^+.
\end{cases}
\end{equation*}
Clearly, $\bm H(0,0,\bm {\check I}^{(2)})=\mathbf 0$, and $D_{\bm {\tilde I}} \bm H(0,0,\bm {\check I}^{(2)})=diag(\delta^{(2)}_j)$, where
\begin{equation}
\delta^{(2)}_j=\begin{cases}\alpha_j(\beta_j-\gamma_j),&j\in H^-,\\
-\alpha_j(\beta_j-\gamma_j),&j\in H^+.
\end{cases}
\end{equation}
Therefore, $D_{\bm {\tilde I}} \bm H(0,0,\bm {\check I}^{(2)})$ is invertible. It follows from the implicit function theorem that
there exist $d_2,\eta_2>0$ and a continuously differentiable mapping
$$
(d_I,\eta)\in[0, d_2]\times [0,\eta_2]\mapsto \bm {\tilde{  I}}(d_I,\eta)=(\tilde {  I}_1(d_I,\eta),\dots,\widetilde{  I}_n(d_I,\eta))^T\in\mathbb R^n
$$
 such that $\bm H(d_I,\eta,\bm {\tilde{  I}}(d_I,\eta))=\mathbf 0$ and
$\bm {\tilde{  I}}(0,0)=\bm {\check I}^{(2)}$.

Taking the derivative of $\bm H(d_I,\eta,\bm {\tilde{  I}}(d_I,\eta))=\bm 0$ with respect to $(d_I,\eta)$ at $(d_I,\eta)=(0,0)$, we have
\begin{equation*}
\begin{cases}
\ds\f{\partial \tilde {  I}_j}{\partial d_I}(0,0)=\ds\f{\sum_{k\in\Omega} L_{jk} {\check I}_k^{(2)}}{\gamma_j-\beta_j}>0,&j=H^{-},\\
\ds\f{\partial \tilde {  I}_j}{\partial d_I}(0,0)=0,&j=H^{+}.\\
\end{cases}
\end{equation*}
Similarly, we have
\begin{equation*}
\begin{cases}
\ds\f{\partial \tilde {  I}_j}{\partial \eta}(0,0)=0,&j=H^{-},\\
\ds\f{\partial \tilde {  I}_j}{\partial \eta}(0,0)=-\ds\f{\gamma_j\alpha_j^2}{(\beta_j-\gamma_j)\alpha_j}<0,&j=H^{+}.\\
\end{cases}
\end{equation*}
Therefore,
$\bm {\tilde{  I}}=\bm{\check I}$. This completes the proof of $(ii)$.
\end{proof}

We also have the following result on an auxiliary problem.
\begin{lemma}\label{d00}
Suppose that $(A_0)$-$(A_3)$ hold and $R_0>1$. Then for any $d\in [0,1)$, the following equation
\begin{equation}\label{auu}
\begin{cases}
d_I\ds\sum_{k\in\Omega} L_{jk} U_k+U_j\left(\beta_j-\gamma_j-\ds\f{\beta_jU_j}{\alpha_j+(1-d)U_j}\right)=0,&j\in\Omega,\\
U_j\ge 0  &j\in\Omega,
\end{cases}
\end{equation}
has a unique strongly positive solution $\bm {\check U}=(\check U_1,\dots,\check  U_n)^T$.
Moreover, $\check U_j$ is monotone decreasing in  $d\in[0,1)$, and
\begin{equation}\label{lims20}
\lim_{d_I\to0} \check U_j=\ds\f{\alpha_j\left(\beta_j-\gamma_j\right)_+}{d\beta_j+(1-d)\gamma_j},\;\;j\in\Omega.
\end{equation}
\end{lemma}

\begin{proof}
We only need to consider the existence and uniqueness of the solution for the case  $d=0$, and the other cases can be proved similar to Lemma \ref{exis}.
Consider the following problem
\begin{equation}\label{auuprof}
\ds\f{d \bar U_j(t)}{dt}=d_I\sum_{k\in\Omega} L_{jk} \bar U_k+\bar U_j\left(\beta_j-\gamma_j-\ds\f{\beta_j\bar U_j}{\alpha_j+\bar U_j}\right),\;\;j\in\Omega.
\end{equation}
Let $\bm{  g}(\bar{\bm U})=\left(  g_1(\bar{\bm U}),\dots,  g_n(\bar{\bm U})\right)^T$ be the vector field corresponding to the right hand side of \eqref{auuprof},
and let $ \Psi_t$ be the  semiflow induced by \eqref{auuprof}.
As in the proof of Lemma \ref{exis}, $\mathbb R^n_+$ is positive invariant with respect to \eqref{auuprof},
$ {\Psi}_t$ is strongly positive and monotone, and $\bm {  g}(\bar{\bm U})$ is strongly sublinear on $\mathbb R^n_+$.
Since $R_0>1$, we have $s\left(d_I L+diag(\beta_j-\gamma_j)\right)>0$. Therefore, by \cite[Corollary 3.2]{ZhaoJing}, we have either
 \begin{enumerate}
 \item [$(i)$] for any initial value $\bm {\bar U}(0)\in\mathbb{R}^n_+\setminus \{\bm 0\}$, the corresponding solution $\bm {\bar U}(t)$ of \eqref{auuprof} satisfies
 $\ds\lim_{t\to \infty}|\bm {\bar U}(t)|=\infty$,
  \end{enumerate}
or alternatively,
 \begin{enumerate}
 \item [$(ii)$] there exists a unique  $\bm {\check U}\gg0$ such
that every solution of \eqref{auuprof} in $\mathbb{R}^n_+\setminus \{\bm 0\}$ converges to $\bm {\check U}$.
 \end{enumerate}
A direct computation implies that, for sufficiently large $M$,
\begin{equation*}
\mathcal V=\left\{\bm U=\left(U_1,\dots,U_n\right)^T\in\mathbb R^n:0\le U_j\le {M\alpha_j},\;j\in\Omega\right\}
\end{equation*}
is positive invariant with respect to \eqref{auuprof}. Therefore, $(i)$ does not hold and $(ii)$ must hold.
 The monotonicity of $\bm {\check U}$  and \eqref{lims20} can be proved similarly as in the proof of Lemmas \ref{exis} and \ref{asym},
respectively. This completes the proof.
\end{proof}
By virtue of Lemmas \ref{asym} and \ref{d00}, we have the following results.
\begin{theorem}\label{dsmain2}
Suppose that $(A_0)$-$(A_3)$ hold and $\gamma_j>0$ for all $j\in\Omega$. Let $(S_1, \dots, S_n$,  $I_1,\dots, I_n)$ be the unique endemic equilibrium of \eqref{patc}-\eqref{N}.  Let
$d_I\to 0$ and $d:={d_I}/{d_S}\to d_0\in[0,\infty]$. Then the following statements hold:
\begin{enumerate}
\item [$(i)$] If $d_0=0$, then
\begin{equation}\label{r1}  S_j\to\ds\f{N\alpha_j}{\ds\sum_{k\in\Omega}\left[\alpha_k+\f{\alpha_k\left(\beta_k-\gamma_k\right)_+}{\gamma_k}\right]},\;\;
  I_j\to\ds\f{ N \f{\ds\alpha_j\left(\beta_j-\gamma_j\right)_+}{\ds \gamma_j}}{\ds\sum_{k\in\Omega}\left[\alpha_k+\f{\alpha_k\left(\beta_k-\gamma_k\right)_+}{\gamma_k}\right]},\;\;j\in\Omega.
\end{equation}
\item [$(ii)$] If $d_0\in(0,\infty)$, then
\begin{equation}\label{r2}
\begin{split}
&  S_j\to\ds\f{ N\left(\alpha_j-\ds\f{d_0\alpha_j\left(\beta_j-\gamma_j\right)_+}{d_0(\beta_j-\gamma_j)_++\gamma_j}\right)}{\ds\sum_{k\in\Omega}\left[\alpha_k+
(1-d_0)\f{\alpha_k\left(\beta_k-\gamma_k\right)_+}{d_0(\beta_k-\gamma_k)_++\gamma_k}\right]},\;\;j\in\Omega,\\
&  I_j\to\ds\f{ N\f{\alpha_j\left(\ds\beta_j-\gamma_j\right)_+}{\ds d_0(\beta_j-\gamma_j)_++\gamma_j}}{\ds\sum_{k\in\Omega}\left[\alpha_k+
(1-d_0)\f{\alpha_k\left(\beta_k-\gamma_k\right)_+}{d_0(\beta_k-\gamma_k)_++\gamma_k}\right]},\;\;j\in\Omega.
\end{split}
\end{equation}
\item [$(iii)$] If $d_0=\infty$, then
\begin{equation}\label{re3}
\begin{split}
&  S_j\to\begin{cases}
\ds\f{N\alpha_j}{\ds\sum_{k\in H^{-}}\alpha_k},&j\in H^{-},\\
0,&j\in H^{+},
\end{cases}\;\;
  I_j\to 0, \;\;j\in\Omega.
\end{split}
\end{equation}
\end{enumerate}
\end{theorem}
\begin{proof}
Let $$\bm {\check U}=(\check U_1,\dots,\check U_n)^T=\bm {\check I}/d=(\check I_1/d,\dots,\check I_n/d)^T,$$
where $\bm {\check I}$ is the unique strongly positive solution of \eqref{au} with $d={d_I}/{d_S}$.
Then $\bm {\check U}$ is the unique strongly positive solution of \eqref{auu}.
It follows from Theorem \ref{mainexis} that
\begin{equation}\label{proofi}
 S_j=\ds\f{d N(\alpha_j-\check I_j)}{\ds\sum_{k\in\Omega}\left[d(\alpha_k-\check I_k)+\check I_k\right]},\;\;
 I_j=\ds\f{ N\check I_j}{\ds\sum_{k\in\Omega}\left[d(\alpha_k-\check I_k)+\check I_k\right]},
\end{equation}
or equivalently,
\begin{equation}\label{proofu}
 S_j=\ds\f{N(\alpha_j-d\check U_j)}{\ds\sum_{k\in\Omega}\left[(\alpha_k-d\check U_k)+\check U_k\right]},\;\;
 I_j=\ds\f{ N \check U_j}{\ds\sum_{k\in\Omega}\left[(\alpha_k-d\check U_k)+\check U_k\right]}.
\end{equation}

$(i)$ Let $\bm {\check U}^{(i)}=(\check U_1^{(i)},\dots, \check U_n^{(i)})$ be the unique strongly positive solution of  \eqref{auu} with $d=d_i$ for $i=1,2$, where $d_1=0$ and $d_2=1/2$. Then by Lemma \ref{d00}, for $d\in (0, 1/2)$ we have
\begin{equation}\label{inu}
\check U_j^{(2)}\le \check U_j\le \check U_j^{(1)}.
\end{equation}
Therefore, if $j\in H^{-}$,
then $$\lim_{(d_I,d)\to(0,0)} \check U_j\le \lim_{d_I\to0}\check U_j^{(1)} =0.$$
Next we consider the case $j\in H^+$. Notice that $\{ \check U_j\}$ is bounded when $d_I$ and $d$ are small. Then for any sequences
$d_I^{(m)}\to0$ and $d^{(m)}\to0$, there are subsequences
$\{d_I^{(m_l)}\}_{l=1}^\infty$ and $\{d^{(m_l)}\}_{l=1}^\infty$ such that the corresponding
solution $\check U_j^{(l)}$ of \eqref{auu} with $d_I=d_I^{(m_l)}$ and $d=d^{(m_l)}$ satisfies
$\ds\lim_{l\to\infty}\check U_j^{(l)}=\check U_j^*$.
It follows from \eqref{inu} that
$\check U_j^*\ge \ds\lim_{d_I\to0}\check U_j^{(2)}>0$. Substituting
$U_j=\check U_j^{(l)}$, $d=d^{(m_l)}$ and $d_I=d_I^{(m_l)}$  into  \eqref{auu} and taking $l\to\infty$ on both sides, we see
that
$$\check U_j^*\left(\beta_j-\gamma_j-\ds\f{\beta_j\check U^*_j}{\alpha_j+\check U^*_j}\right)=0,$$
which implies that
\begin{equation}
\lim_{(d_I,d)\to(0,0)} \check U^*_j=\ds\f{\alpha_j\left(\beta_j-\gamma_j\right)_+}{\gamma_j},\;\;j\in\Omega.
\end{equation}
This, combined with \eqref{proofu}, implies  \eqref{r1}.

$(ii)$ Let $\bm {\check I}^{(i)}=(\check I^{(i)}_1,\dots,\check I^{(i)}_n)^T$ be the unique strongly positive
solution of \eqref{exis} with $d=d_i$ for $i=1,2$, where $d_1=d_0/2$ and $d_2=2d_0$. We see from Lemma \ref{exis} that, for $d\in[d_0/2,2d_0]$,
$$\check I_j^{(1)}\le\check I_j\le \check I_j^{(2)}\;\;\text{for}\;\;i\in\Omega.$$
Therefore, if $j\in H^{-}$,
then $$\lim_{(d_I,d)\to(0,d_0)} \check I_j\le \lim_{d_I\to0}\check I_j^{(2)} =0.$$
Next  we consider the case  $j\in H^+$. Note that $\{ \check I_j\}$ is bounded. Then for any sequences
$d_I^{(m)}\to0$ and $d^{(m)}\to d_0$, there are subsequences
$\{d_I^{(m_l)}\}_{l=1}^\infty$ and $\{d^{(m_l)}\}_{l=1}^\infty$ such that the corresponding
solution $\check I_j^{(l)}$ of \eqref{au} with $d_I=d_I^{(m_l)}$ and $d=d^{(m_l)}$  satisfies
$\ds\lim_{l\to\infty}\check I_j^{(l)}=\check I_j^*$.
It follows from \eqref{lims} that
$\check I_j^*\ge \ds\lim_{d_I\to0}\check I_j^{(1)}>0$. Substituting
$I_j=\check I_j^{(l)}$, $d=d^{(m_l)}$ and $d_I=d_I^{(m_l)}$  into \eqref{au} and taking  $l\to\infty$ on both sides, we see
that
$$\check I_j^*\left(\beta_j-\gamma_j-\ds\f{\beta_j\check I^*_j}{d_0(\alpha_j-\check I^*_j)+\check I^*_j}\right)=0,$$
which implies that
\begin{equation}
\lim_{(d_I,d)\to(0,\infty)} \check I^*_j=\ds\f{d_0\alpha_j\left(\beta_j-\gamma_j\right)_+}{d_0(\beta_j-\gamma_j)_++\gamma_j},\;\;j\in\Omega.
\end{equation}
This, combined with \eqref{proofi}, implies  \eqref{r2}.

$(iii)$ By Lemma \ref{asym}, we have
\begin{equation}
\lim_{(d_I,d)\to(0,\infty)} \check I_j=\begin{cases}
0,&j\in H^{-},\\
\alpha_j,&j\in H^+.
\end{cases}
\end{equation}
This, together with \eqref{proofi}, implies \eqref{re3}.
\end{proof}

\section{An example}
In this section, we give an example to illustrate the results in Sections 2-3.
Here we use the star graph (Fig. \ref{fig1}) as the migration pattern between patches, i.e.
the population distribution entails a central deme and $n-1$ colonies extending along rays \cite{Karlin}.
\begin{figure}[htbp]
\centering\includegraphics[width=0.5\textwidth]{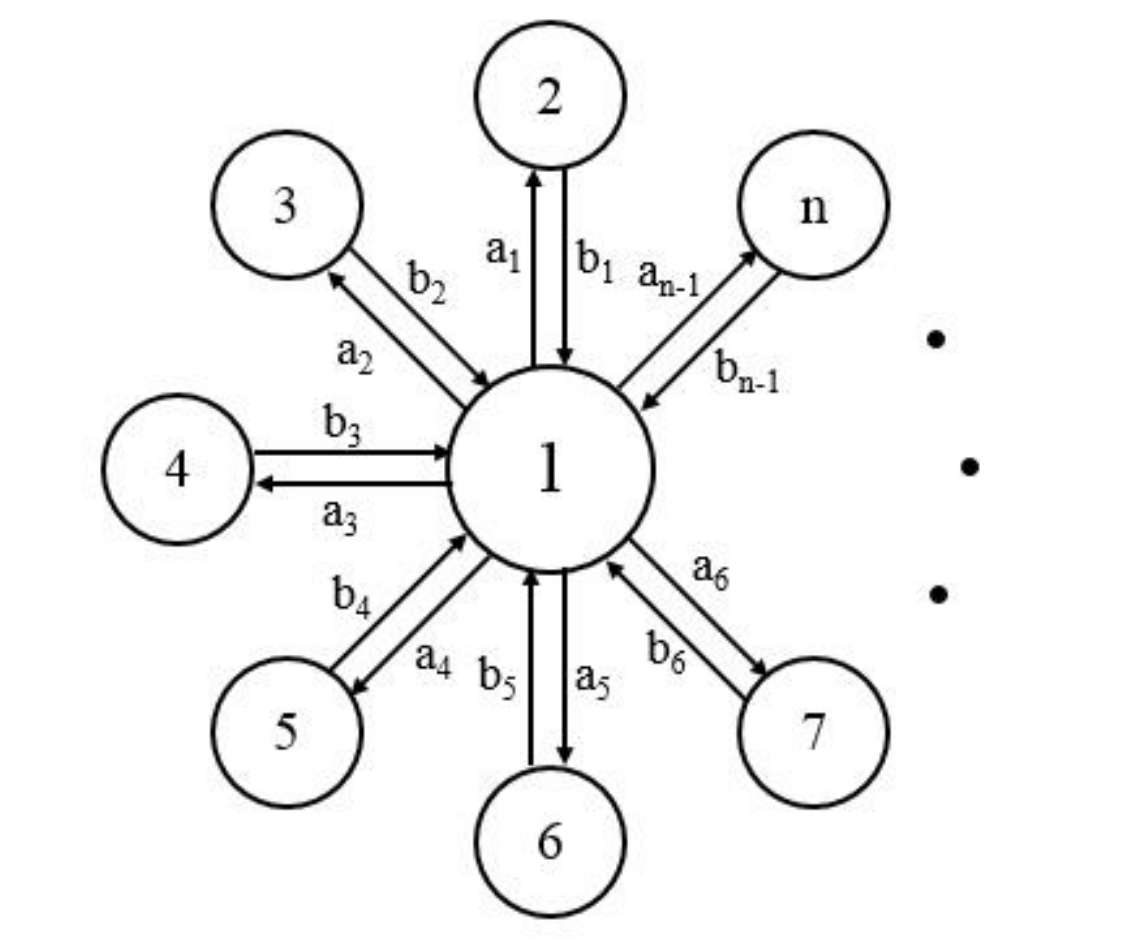}
\caption{The star migration graph.
  \label{fig1}}
\end{figure}
Then the connectivity matrix $L$ is an $n\times n$ $(n\ge2)$ matrix:
\begin{equation}
L= \begin{pmatrix}
-\ds\sum_{i=1}^{n-1}a_i & b_1 & b_2&b_3 &\cdots & b_{n-1} \\
   a_1   &-b_1 & 0 &0&\cdots & 0 \\
   a_2   & 0 &-b_2&0&\cdots & 0\\
   a_3   &0  &0   &-b_3&\cdots&0\\
  \vdots&\vdots&\vdots &\vdots & \ddots & \vdots  \\
  a_{n-1}&0  & 0 &0& \cdots & -b_{n-1}
 \end{pmatrix}.
\end{equation}
Denote $r_i={a_i}/{b_i}$ for
 $i=1,\dots,n-1$. A direct computation gives
 $$\bm \alpha=\left(\ds\f{1}{1+s},\ds\f{r_1}{1+s}, \dots, \ds\f{r_{n-1}}{1+s}\right),$$
 where $s=\ds\sum_{i=1}^{n-1} r_i$.
Here we assume:
\begin{enumerate}
\item [($A$)]
$H^+=\{1,2\}$ and $H^{-}=\{3,\dots,n\}$.
\end{enumerate}
 This assumption means that
 patch-$1$ and patch-$2$ are of high-risk, and all others are of low-risk.

 For this example, we can compute
 \begin{equation*}
\check I_j^{(0)}=\begin{cases}
\alpha_j,&j =1,2,\\
\ds\f{d_Ia_{j-1}\alpha_1}{d_I b_{j-1}+\gamma_j-\beta_j}, &j=3,\dots,n,
\end{cases}
\end{equation*}
and
\begin{equation*}
\begin{split}
h_1(d_I)=&d_I\left[\left(-\sum_{k=1}^{n-1}a_k\right)\alpha_1+b_1\alpha_2+\sum_{k=3}^n\ds\f{d_I\alpha_1a_{k-1}b_{k-1}}{d_Ib_{k-1}+\gamma_k-\beta_k}\right]
+\alpha_1(\beta_1-\gamma_1).\\
h_2(d_I)=&\alpha_2(\beta_2-\gamma_2)>0
\end{split}
\end{equation*}
 It follows from Proposition \ref{prop3.10} that
$h_1(d_I)$
is strictly decreasing and satisfies
\begin{equation}\label{hlimit}
\lim_{d_I\to0}h_1(d_I)=\alpha_1(\beta_1-\gamma_1)>0,\;\;\text{and}\;\;\lim_{d_I\to\infty}h_1(d_I)=\alpha_1(\beta_1-\gamma_1)+\ds\sum_{k=3}^n\alpha_k(\beta_k-\gamma_k).
\end{equation}
By Lemma \ref{monos}, we have
that
\begin{equation*}
\begin{split}
&\lim_{d_I\to 0}s\left(d_IL+diag(\beta_j-\gamma_j)\right)=\ds \max_{1\le k=\le n}(\beta_k-\gamma_k)>0,\\
&\lim_{d_I\to 0}s\left(d_IL+diag(\beta_j-\gamma_j)\right)=\ds\sum_{k=1}^n\alpha_k(\beta_k-\gamma_k).
\end{split}
\end{equation*}
Since $s\left(d_IL+diag(\beta_j-\gamma_j)\right)$ has the same sign as $R_0-1$ and is strictly decreasing for $d_I$, we have the following result.
\begin{proposition}\label{prostar}
Suppose  $a_k,b_k>0$ for $k=1,\dots,n-1$ and $(A)$ holds. Then the following statements hold:
\begin{enumerate}
    \item [$(i)$] If $\ds\sum_{k=1}^n\alpha_k(\beta_k-\gamma_k)>0$, then $R_0>1$ for any $d_I>0$. Moreover,
\begin{enumerate}
\item [$(i_1)$] if $\alpha_1(\beta_1-\gamma_1)+\ds\sum_{k=3}^n\alpha_k(\beta_k-\gamma_k)\ge 0$, then  $J^+=H^+$ and $J^-=H^-$ for any $d_I>0$;
\item [$(i_2)$] if $\alpha_1(\beta_1-\gamma_1)+\ds\sum_{k=3}^n\alpha_k(\beta_k-\gamma_k)< 0$, then there exists a unique $d_I^{**}$ such that
$ h_1(\tilde d_I)=0$, and
 $J^+=H^+$ and $J^-=H^-$ for $0<d_I<d_I^{**}$, and $J^{+}=\{1\}$ and $J^{-}=\{2,\dots,n\}$, or $J^{+}=\{2\}$ and $J^{-}=\{1,3,\dots,n\}$ for $d_I> d_I^{**}$.
 \end{enumerate}
\item [$(ii)$] If $\ds\sum_{k=1}^n\alpha_k(\beta_k-\gamma_k)<0$, then $\alpha_1(\beta_1-\gamma_1)+\ds\sum_{k=3}^n\alpha_k(\beta_k-\gamma_k)< 0$, and there exists $d_I^*>0$ such that $R_0>1$ for $d_I<d_I^*$ and $R_0<1$ for $d_I>d_I^*$. Moreover,
\begin{enumerate}
\item [$(ii_1)$] if $d_I^{**}\ge d_I^*$,  where $d_I^{**}$ is defined as in $(i_2)$, then  $J^+=H^+$ and $J^-=H^-$ for $d_I<d_I^*$;
\item [$(ii_2)$] if $d_I^{**}< d_I^*$, then $J^+=H^+$ and $J^-=H^-$ for $d_I< d_I^{**}$; and $J^{+}=\{1\}$ and $J^{-}=\{2,\dots,n\}$, or $J^{+}=\{2\}$ and $J^{-}=\{1,3,\dots,n\}$ for $d_I\in(d_I^{**},d_I^*)$.
 \end{enumerate}
\end{enumerate}
\end{proposition}
\begin{remark}
From Proposition \ref{prostar}, we see that case $(i_1)$ could hold when $\beta_1-\gamma_1$ is sufficiently large; case $(i_2)$ could hold
when $\beta_1-\gamma_1$ is sufficiently small but $\beta_2-\gamma_2$ is sufficiently large; and if both $\beta_1-\gamma_1$ and
$\beta_2-\gamma_2$ are sufficiently small, case $(ii_1)$ or $(ii_2)$ could occur.
\end{remark}
The asymptotic profile of the endemic equilibrium as  $d_I\rightarrow 0$ can also be obtained from Theorem \ref{dsmain2}.
To further illustrate our results, we give a numerical example of star graph with $n=4$. Let
\begin{equation*}
L=
    \begin{pmatrix}
    -6 & 1&1&1\\
    1&-1&0&0\\
    2&0&-1&0\\
    3&0&0&-1
    \end{pmatrix}.
\end{equation*}
Then $\bm \alpha=(1/7, 1/7, 2/7, 3/7)$.
We choose $\beta_1=3, \beta_2=4, \beta_3=1, \beta_4=1$,$\gamma_1=1, \gamma_2=1, \gamma_3=2, \gamma_4=7$  such that $H^+=\{1, 2\}$ and $H^-=\{3, 4\}$, and $N=100$. Theorem \ref{samemono} states that $R_0$ is strictly deceasing in $d_I$ with
$$
\lim_{d_I\to 0}R_0=\max\left\{\frac{\beta_j}{\gamma_j}: j\in\Omega\right\}=4\ \ \ \text{and}\ \ \ \lim_{d_I\to \infty}R_0=\frac{\sum_{j\in\Omega}\alpha_j\beta_j}{\sum_{j\in\Omega}\alpha_j\gamma_j}=\frac{4}{5}.
$$
In Figure \ref{R0}, we plot $R_0$ as a function of $d_I$, which confirms Theorem \ref{samemono}. Here, $R_0-1$ changes sign at $d_I^*\approx 8.478$.

\begin{figure}[htbp]
\centering\includegraphics[width=0.8\textwidth]{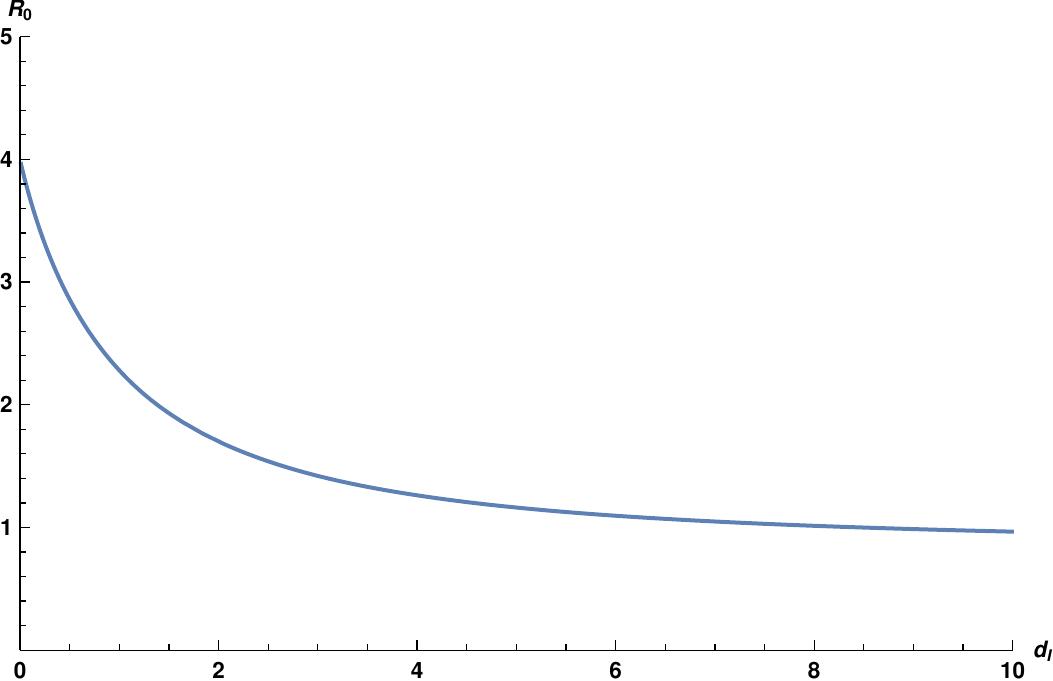}
\caption{The basic reproduction number $R_0$ as a function of $d_I$.
  \label{R00}}
\end{figure}

Then we compute $h_j(d_I)$, $j=1, 2$. By Proposition \ref{prop3.10}, $h_j(d_I)$ is constant or strictly decreasing in $d_I$. By \eqref{hlimit}, we expect
$$
\lim_{d_I\rightarrow 0} h_1(d_I)=\frac{2}{7}\ \ \ \text{and} \ \ \ \lim_{d_I\rightarrow \infty} h_1(d_I)=-\frac{6}{7},
$$
and $h_2(d_I)=3/7$ for all $d_I>0$. These results are confirmed by Figure \ref{R0}. Moreover, we have $h_1(0.549)\approx 0$. By $\sum_{j\in\Omega}\alpha_j(\beta_j-\gamma_j)=-3/7<0$ and Proposition \ref{prostar}(ii), we expect that the profile of the endemic equilibrium changes at $d_I^{**}\approx 0.549$.
\begin{figure}[htbp]
\centering\includegraphics[width=0.8\textwidth]{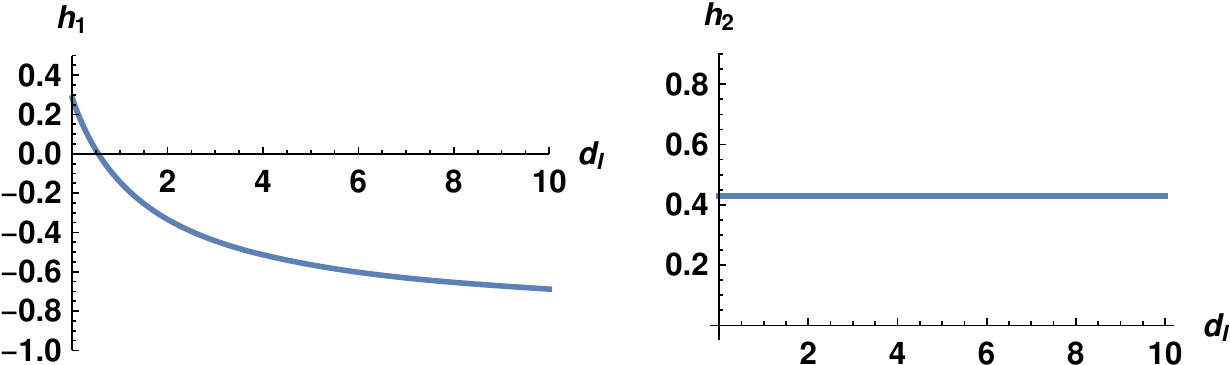}
\caption{The graph of $h(d_I)$.
  \label{R0}}
\end{figure}
In Figure \ref{Fig2}, we plot the $S$ component of the endemic equilibrium as $d_S\rightarrow 0$, where $S_j^*=\ds\lim_{d_S\rightarrow 0} S_j$ for $j=\{1, 2, 3, 4\}$. From the figure, we see that $J^+=\{1, 2\}$ and $J^-=\{3, 4\}$ for $d_I\in (0, d_I^{**})$ and $J^+=\{2\}$ and $J^-=\{1, 3, 4\}$ for $d_I\in ( d_I^{**}, d_I^{*})$.

\begin{figure}[htbp]
\centering\includegraphics[width=\textwidth]{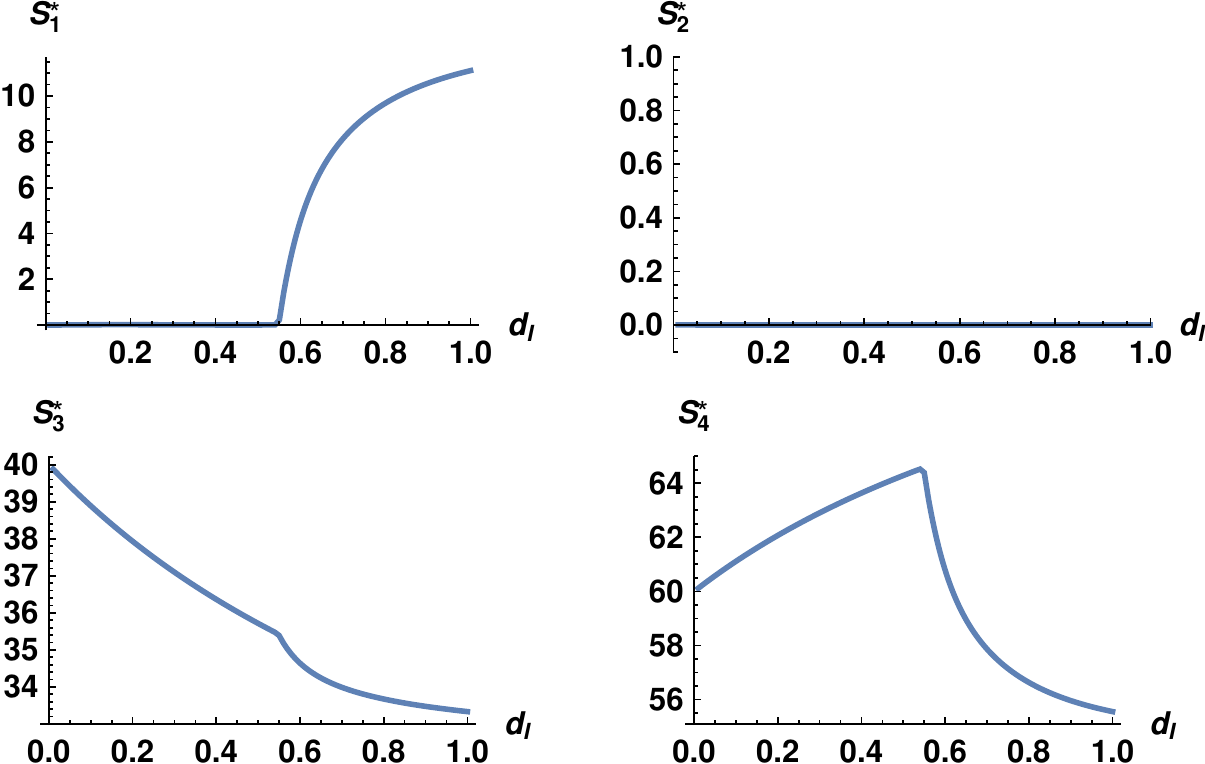}
\caption{The $S$ component of the endemic equilibrium as $d_S\rightarrow 0$.
  \label{Fig2}}
\end{figure}

\end{document}